\documentclass{amsart}

\usepackage[
english]{babel}

\usepackage{amsmath,amsfonts,amssymb}
\usepackage{enumerate}
\usepackage{mathtools}
\usepackage[latin1]{inputenc}
\usepackage[colorlinks=true]{hyperref}
\hypersetup{urlcolor=blue, citecolor=blue}

\usepackage{color,tikz}

\newtheorem{thm}{Theorem}[section]
\newtheorem{lem}[thm]{Lemma}
\newtheorem{prop}[thm]{Proposition}

\newtheorem{Def}[thm]{Definition}
\newenvironment{Proof}{\noindent {\it Proof :}}{\hfill $\square$}
\newtheorem{rem}[thm]{Remark}

\def\R{\mathbb{R}}
\def\N{\mathbb{N}}
\def\A{\mathbb{A}}
\def\M{\mathbb{M}}
\def\D{\mathbb{D}}
\def\T{\mathbb{T}}
\def\grad{\boldsymbol{\nabla}}
\def\p{\partial}
\def\s{\sigma}
\def\be{\begin{equation}}
\def\ee{\end{equation}}
\def\O{\Omega}
\def\Tt{\mathcal{T}}
\def\dt{{\Delta t}}

\def\d{{\rm d}}
\def\x{{\bf x}}
\def\y{{\bf y}}
\def\n{{\bf n}}
\def\div{{\rm div}}

\def\eps{\epsilon}
\def\1{{\bf 1}}
\def\un#1{\underline{#1}}
\def\ov#1{\overline{#1}}
\def\a{\alpha}
\def\k{\kappa}
\def\nn{\nonumber}
\def\BV{{\rm BV}}
\def\h#1{\widehat{#1}}

\def\Aa{\mathcal A}
\def\Bb{\mathcal B}
\def\Cc{\mathcal C}
\def\Dd{\mathcal D}
\def\Ee{\mathcal E}

\def\Mm{\mathcal M}
\def\Nn{\mathcal N}
\def\Oo{\mathcal O}
\def\Pp{\mathcal P}
\def\Uu{\mathcal U}
\def\Vv{\mathcal V}
\def\nn{\nonumber}
\def\bv{{\boldsymbol v}}
\def\bP{{\boldsymbol P}}
\def\bbK{{\mathbb K}}
\def\bx{{\boldsymbol x}}
\def\bphi{{\boldsymbol \varphi}}
\def\u{{\boldsymbol u}}
\def\g{{\boldsymbol g}}

\title[Nonlinear time compactness result and applications]{A nonlinear time compactness result\\
 and applications to discretization\\ of degenerate parabolic-elliptic PDEs}
\author{B. Andreianov}
\address{
Boris Andreianov ({\tt boris.andreianov@univ-fcomte.fr})
\begin{enumerate}
\item Laboratoire de Math{\'e}matiques de Besan{\c c}on,
CNRS UMR 6623 \\
Universit{\'e} de Franche-Comt{\'e} \\
16 route de Gray, 25030 Besan{\c c}on Cedex, France
\end{enumerate}
}

\author{C. Canc\`es \and A. Moussa}
\address{
Cl\'ement Canc\`es ({\tt cances@ljll.math.upmc.fr}), Ayman Moussa ({\tt moussa@ljll.math.upmc.fr})
\begin{enumerate}
\item Sorbonne Universit\'es, UPMC Univ Paris 06, UMR 7598, Laboratoire Jacques-Louis Lions, F-75005, Paris, France
\item CNRS, UMR 7598, Laboratoire Jacques-Louis Lions, F-75005, Paris, France
\end{enumerate}
}

\begin{document}

\maketitle

\begin{abstract}
We propose a discrete functional analysis result suitable
for proving compactness in the framework of fully discrete
approximations of strongly degenerate parabolic problems.
It is based on the original exploitation of a
result related to compensated compactness rather than
on a classical estimate on the space and time translates in the spirit of Simon (Ann. Mat. Pura Appl. 1987).
Our approach allows to handle various numerical
discretizations both in the space variables
and in the time variable. In particular, we can cope
quite easily  with variable time steps and with multistep
time differentiation methods like, e.g., the
backward differentiation formula of order~2 (BDF2) scheme.
We illustrate our approach by proving the
convergence of a two-point flux Finite Volume
in space and BDF2 in time approximation of the
porous medium equation.
\end{abstract}

\section{Introduction}

\smallskip
There exists a large variety of numerical strategies for discretization of evolution PDEs.
Proofs of convergence of many different numerical schemes often take the following standard
itinerary (see, e.g., \cite{EGH00}). Given a PDE, discrete equations of the scheme are rewritten under a form reminiscent of the weak formulation of the continuous problem; stability estimates are obtained, which ensure bounds in appropriate (possibly discretized) functional spaces; eventually, sufficiently strong compactness arguments permit to pass to the limit in the discrete weak formulation.

\smallskip
In many applications, including degenerate parabolic equations of various kinds, the question of strong compactness in $L^p$ spaces of sequences of approximate solutions is a cornerstone of such proofs. While ``space compactness'' is usually obtained by suitable \emph{a priori} estimates
of the discrete gradients involved in the equation, ``time compactness'' is often obtained by explicitly estimating
$L^2$ or $L^1$ time translates in the spirit of Alt and Luckhaus \cite{AL83}. This step is equation-dependent,
and has to be reproduced for each problem.
We refer to \cite{EGH00,EGHMichel} for the main ingredients of this already classical argument
used in a number of subsequent works, to \cite{AndrBendKarl10,Andr-FVCA6Proc} for some refinements,
and to \cite[\S3]{BM13} for a shortened version of the argument.

\smallskip
 The same question of time compactness often arises in existence analysis for PDEs, in the continuous framework.
 Along with the technique of \cite{AL83}, there exist several ready-to-use results that yield space-time precompactness of a sequence of (approximate) solutions $(u_n)_n$.
 They are based on the two following ingredients:\\[3pt]
 $(A)$  estimates in sufficiently narrow Bochner spaces\\
 \hspace*{12pt} ensuring uniform in $n$ bounds on space translates of the sequence $(u_n)_n$;\\[1pt]
 $(B)$ some very weak (in the space variable) estimates on the sequence $(\partial_t u_n)_n$.\\[3pt]
Then, different arguments permit to derive from $(A)$ and $(B)$ uniform in $n$ estimates of time translates of $(u_n)_n$ and conclude that $(u_n)_n$ is compact in the appropriate space (e.g., as a consequence of the Fr\'echet-Kolmogorov compactness criterion for $L^p$ spaces).
 This kind of result, in the abstract linear setting, is often called Aubin-Lions-Simon lemma \cite{Aubin,Lions,Simon}, but the version we are interested in is also related to the early nonlinear version of the argument due to Dubinskii \cite{Dubinskii} (see also recent references \cite{BarrettSuli,ChenLiu,ChenJungelLiu}) and to the more recent formulation of Ma\^{\i}tre \cite{Maitre}. Another related argument of nonlinear kind is due to Kruzhkov \cite{Kruzhkov}. Further improvements were obtained  by Amann in \cite{amann} for a refined scale of spaces (including Besov spaces for instance), and broached by Roub{\'{\i}}{\v{c}}ek in a rather general setting, see \cite{roub}. One observes that several closely related results co-exist, but the precise assumptions and conclusions of there results differ. Therefore,  one can see the combination of properties $(A)\&(B)$ as a ``time compactness principle'', which can be made precise upon choosing a suitable functional framework or a suitable form of the estimates $(A)$ and $(B)$ (as one illustration, we refer to Emmrich and Thalhammer \cite{EmmrichThalhammer} where $(B)$ is formulated as a fractional time derivative estimate).

\smallskip
 Further, discrete variants of different time compactness results have been already proved both for many concrete applications (mainly in the context of finite element or finite volume methods) and in abstract form: we refer in particular to Eymard et al. \cite{EGH00,EGHMichel} for Alt-Luckhaus kind technique for concrete applications, to Gallou\"et and Latch\'e \cite{GallouetLatche} for a discrete Aubin-Lions-Simon lemma, to Andreianov et al. \cite{ABRB11,AndrBendHubert} for a discrete Kruzhkov lemma, and to Dreher and J\"ungel \cite{DreherJungel} (see also \cite{ChenJungelLiu}) for a discrete Dubinskii argument with uniform time stepping.

 \smallskip
 The new result we are intended to present is based upon the technique of \cite{Moussa14} which carries on rather easily to the discrete case. Our result is formulated as the ready-to-use Theorem~\ref{thm:main} applicable to a large variety of numerical schemes (including variable time step and multi-step methods in time). It can be applied to a wide variety of strongly degenerate nonlinear parabolic equations.

  \subsection{Compactness arguments for degenerate parabolic PDEs}

In what follows,  $\O$ is a bounded open subset of $\R^d$, while $T>0$ is an arbitrary finite time horizon.
The cylinder $\O\times(0,T)$ is denoted by $Q_T$.

 Depending on the type of degeneracy of the underlying PDE, some of the aforementioned time compactness lemmas or techniques can be applied and the others fail to fit the structure of nonlinearities involved in the equation. To be specific, consider as the fundamental example the family of degenerate parabolic equations
 \begin{equation}\label{eq:General-DegenParabolic}
   \p_t u = \Delta v + \text{LOT}, \;\; u\in\beta(v)
 \end{equation}
where $\beta\subset \R^2$ is a maximal monotone graph and $\text{LOT}$ represent some lower-order terms, e.g., of convection and reaction kind. Definition and basic properties of monotone graphs are recalled in \S\ref{subsec:des_cont} for readers who are not familiar with this notion. The graph $\beta$ can contain vertical and horizontal segments, which leads to elliptic-parabolic and (in presence of first-order convection terms $\text{LOT}$) to parabolic-hyperbolic strong degeneracy, respectively.  In fact, most of the different time compactness arguments were developed for applications of the kind \eqref{eq:General-DegenParabolic}, with possible coupling to other equations.

  Let us assume for simplicity  that
we have a sequence of solutions to \eqref{eq:General-DegenParabolic} with uniform in $n$ control of $\{v_n\}_n$ in $L^2(0,T;H^1_0(\Omega))$ and of $\{\p_t u_n\}_n$ in $L^2(0,T;H^{-1}(\Omega))$.

\smallskip
\noindent$\bullet$ Firstly, the Aubin-Lions lemma \cite{Aubin,Lions} can be applied in this situation, provided the graph $\beta$
is bi-Lipschitz, i.e., if it is Lipschitz continuous with Lipschitz continuous inverse. This is the non-degenerate, uniformly parabolic case. Even in the power case $\beta(v)=\text{sign}(v) |v|^{\alpha}$, the degeneracy at zero ($\alpha\in(0,1)$, the porous medium equation) or the singularity at zero ($\alpha>1$, the fast diffusion equation) do not permit to apply the Aubin-Simon lemma.
The classical generalization by Simon \cite{Simon} of the Aubin-Lions lemma gives more precise compactness information under less restrictive estimates of $\{v_n\}_n$ and of $\{\p_t u_n\}_n$, but it does not help to overcome
the possible degeneracy of $\beta$. Degenerate cases require arguments of nonlinear kind.

 \smallskip
\noindent$\bullet$
 The elliptic-parabolic degenerate case (i.e., the case where $\beta$ is a continuous map) can be handled using Ma\^{\i}tre's lemma~\cite{Maitre} or the Kruzhkov's one~\cite{Kruzhkov,AndrGutnicWittbold,Andr-FVCA6Proc}. Application of each of these tools requires some additional assumptions such as the uniform continuity of $\beta$ or the boundedness of $\|v_n\|_\infty$.
 The difficulty in application of Ma\^{\i}tre's lemma~\cite{Maitre} consists in the choice of appropriate functional setting according to the behavior of the nonlinearity $\beta$. It may require restrictions on the behavior of $\beta$ at infinity and introduction of special functional spaces, e.g. of the Orlicz kind.

\smallskip
\noindent$\bullet$ In what concerns discrete versions of the above general lemmas,
the proof presented in~\cite{Kruzhkov} is particularly simple to adapt to the discretized setting,
indeed, it uses only the most natural $L^1$ norm for discrete solutions.
Adaptation to the discrete setting of the linear compactness lemmas of \cite{Aubin,Lions,Simon} is presented in \cite{GallouetLatche}; it requires the construction of the adequate discrete spaces and an ingenious reformulation of the assumptions. It is feasible that also the nonlinear compactness lemma of \cite{Maitre} can be adapted to discretized setting using the approach of \cite{GallouetLatche}, but this question is beyond our scope.

\smallskip
\noindent$\bullet$
Further, the elliptic-parabolic case (i.e., the case where $\beta$ is a continuous map) and also the parabolic-hyperbolic
case (i.e., the case where $\beta^{-1}$ is a continuous map) can be dealt with using the Alt-Luckhaus approach.
It can be roughly described as estimating the integral over $Q_T$ of the product
\begin{equation}\label{eq:AL-estimatedquantity}
\Bigl(u_n(\x,t+\tau)-u_n(\x,t)\Bigr)\Bigl(v_n(\x,t+\tau)-v_n(\x,t) \Bigr)
\end{equation}
by a uniformly vanishing, as $\tau\to 0$, modulus of continuity.
Calculations leading to such estimate of \eqref{eq:AL-estimatedquantity} use the variational structure of the equations and the Fubini theorem; although they are standard, they appear to be equation-dependent and (in the discretized
setting) scheme-dependent.

Starting from \eqref{eq:AL-estimatedquantity} and estimates of $\nabla v_n$, conclusions about relative compactness in $L^2(Q_T)$ of $\{u_n\}_n$ (in the elliptic-parabolic case), respectively of $\{v_n\}_n$ (in the parabolic-hyperbolic case) can be derived,  provided $\beta$ (respectively, $\beta^{-1}$) is assumed to be Lipschitz continuous. Mere uniform continuity of $\beta$ (respectively, $\beta^{-1}$) is enough for $L^1$ compactness, see \cite{AndrBendKarl10}.

Adaptation of estimates of the kind \eqref{eq:AL-estimatedquantity} to the discrete setting became a classical trend, starting from \cite{EGH00}. Yet the use of Fubini argument in the time-discretized setting brings lengthy technicalities, that are reproduced in an impressive number of papers dealing with convergence of finite volume approximations to various parabolic problems.
Therefore, our goal is to provide a black-box avoiding these computations.

\subsection{Description of our approach in the continuous setting.}\label{subsec:des_cont}
To give an idea, in this paragraph we argue on equations \eqref{eq:General-DegenParabolic}.
Let $(u_n,v_n)$ be (approximate) solutions of the problem.
Given the structure of the equation, one can require that $(v_n)_n$
is controlled in $L^2(0,T;H^1(\Omega))$ and $(\p_t u_n)_n$ is controlled in the dual space.
In the case $\beta=\text{Id}$, i.e. $u_n=v_n$ (with an immediate extension to
bi-Lipschitz $\beta$), this kind of assumptions is the basis of the Aubin-Lions-Simon
kind lemmas; they are exploited for estimating space and time translates of the solutions,
in order to apply the Fr\'echet-Kolmogorov compactness argument.
The idea of the compactness lemma we prove in this paper consists in justifying,
under the same kind of assumptions, the relation
\begin{equation}\label{eq:LemmdeBase}
 u_n v_n \rightharpoonup uv \;\;\text{in $\Dd'(Q_T)$}
\end{equation}
(up to extraction of a subsequence), where $u$,$v$ are the respective weak limits of $u_n$,$v_n$
(say, in $L^2(Q_T)$). Then we exploit this weak convergence property thanks to the monotonicity of $\beta$.

Motivated by application in porous media flows (see e.g. \S\ref{ssec:Richards}), we introduce a sequence
$\left(\omega_n\right)_n$ of $L^\infty$ weights with $L^\infty$ inverse converging almost everywhere to some
limit $\omega$. More precisely, we assume that there exists $\un \omega, \ov \omega >0$ such that
\be\label{eq:poids-conv}
\un \omega \le \omega_n(\x) \le \ov \omega \quad\text{ and }
\quad \omega_n(\x) \underset{n\to\infty}\longrightarrow \omega(\x) \quad \text{ for a.e. } \x \in \O.
\ee

Let us first state a slightly modified version of the key technical tool \cite[Prop.\,1]{Moussa14}. This result is reminiscent of the celebrated framework of Tartar-Murat (see \cite{tartar,murat}), even though to our knowledge, there is no direct relation between the general theory of compensated compactness and the one of \cite{Moussa14}. To simplify the statement, we exclude the cases $q=1$ and $q=\infty$, where weak convergence should be replaced by the weak-* convergence (for $(u_n)_n$, if $q=\infty$; for $(v_n)_n$, if $q=1$).

\begin{prop}\label{th:Moussa-compcomp}
Let $\left(\omega_n\right)_{n\ge 1} \subset L^\infty(\O)$ be such that~\eqref{eq:poids-conv} holds.
Let $q\in(1,\infty)$ and $p\in[1,\infty)$, and let $\a > \frac{pd}{p+d}$.
Assume that $(u_n)_n$, $(v_n)_n$ are two sequences of measurable functions on $Q_T$ such that
$(u_n)_n$ is bounded in $L^q((0,T);W^{1,\a}(\O))$ and $(v_n)_n$ is bounded in $L^{q'}((0,T);L^{p'}(\O))$.
Up to extraction of a subsequence, we can assume that $u_n$ and $v_n$ weakly converge to $u$, $v$ respectively in $L^q((0,T);L^{p}(\O))$ and $L^{q'}((0,T);L^{p'}(\O))$.
Assume that, in addition, $(v_n)_n$ verifies the following uniform ``weighted'' weak estimate:
\begin{equation}\label{eq:dual-estimate}
\int_0^T\!\!\int_\O \omega_nv_n\partial_t \varphi \leq C\| \nabla_x \varphi \|_{\infty}, \qquad \forall \varphi\in \mathcal D(Q_T).
\end{equation}
Then
\begin{equation}\label{eq:compcomp}
\int_0^T\!\!\int_\O 
\omega_n u_n v_n \varphi \underset{n\to\infty}\longrightarrow
\int_0^T\!\!\int_\O 
\omega u v \varphi, \qquad   \forall\,\varphi\in \mathcal D(Q_T).
\end{equation}
\end{prop}
Observe the two main differences with  \cite[Prop.\,1]{Moussa14}. Firstly, weights $\omega_n$ are introduced, which may be useful in PDEs modeling flows in heterogeneous porous media. For Proposition \ref{th:Moussa-compcomp}, the case of general $\omega_n$ follows from the particular one $\omega_n=1$ (one can replace $v_n$ by $\omega_n v_n$), which is the one considered in \cite{Moussa14}.Nevertheless the presence of general $\omega_n\neq 1$ will be more intricate to handle in the proof of Proposition \ref{prop:Young} (see below), this is why we keep track of $\omega_n$ here. The second difference is the bound \eqref{eq:dual-estimate} which appears stronger than the one assumed in \cite{Moussa14}, in which $(v_n)_n$  is only required to be bounded in some $\BV((0,T);H^{-m}(\O))$ space, where $m$ can be as large as needed. However, in our case we want to limit our considerations to $m=1$, because we will focus on numerical approximations and the information on higher-order in space discrete derivatives can be delicate to obtain. In this respect, the assumption \eqref{eq:dual-estimate} is the weakest assumption corresponding to $m=1$.

\vspace{2mm}

Before going further, let us recall the definition and a few basic properties of maximal monotone graphs. For more details, see for instance \cite{brezis}.
\begin{Def}
A \emph{monotone graph} on $\R$ is a map $\beta$ from $\R$ to the set $\mathcal{P}(\R)$ of all subsets of $\R$ such that $(y_1-y_2)(x_1-x_2)\geq 0$ for any $(x_1,x_2)\in \R^2$ and any $(y_1,y_2)\in \beta(x_1)\times\beta(x_2)$. It is said to be \emph{maximal monotone} if it admits no non-trivial monotone extension. If $\beta$ is a monotone graph, we denote by $\beta^{-1}:\R\rightarrow\mathcal{P}(\R)$ the map defined by $y\in\beta^{-1}(x) \Leftrightarrow x \in\beta(y)$.
\end{Def}
\begin{prop}\label{prop:monograph}
Given a  monotone graph $\beta$, the following statements are equivalent:
\begin{enumerate}[(i)]
\item $\beta$ is a maximal monotone graph.
\item $\beta^{-1}$ is a maximal monotone graph.
\item for all $\lambda>0$, $(\textnormal{Id}+\lambda \beta)^{-1}:\R\to\R$ is a (single valued) non-expansive mapping.
\end{enumerate}
\end{prop}

The ``compensated compactness'' feature \eqref{eq:compcomp} can be exploited in particular in the following context.
\begin{prop}\label{prop:Young}
  Let $(u_n)_n$, $(v_n)_n$ be two sequences of real valued functions on $Q_T$ weakly converging to
  limits $u$, $v$ in $L^1(Q_T)$. We assume that the limits $u$ and $v$ satisfy $uv \in L^1(Q_T)$.
  Assume that in addition, \eqref{eq:compcomp} holds.
 Let $\beta$ be a maximal monotone graph with $0\in\beta(0)$.
 If for all $n$, $v_n\in \beta(u_n)$ a.e. in $Q_T$, then $v\in \beta(u)$ a.e. on $(0,T)\times\O$. Moreover, up
 to the extraction of an unlabeled subsequence,
 for almost every $(\x,t) \in Q_T$, either $u_n(\x,t) \to u(\x,t)$ or $v_n(\x,t) \to v(\x,t)$ and
\begin{enumerate}[(i)]
   \item $v_n\to v$ a.e. in $Q_T$ if $\beta$ is single valued;
   \item $u_n\to u$ a.e. in $Q_T$ if $\beta^{-1}$ is single valued.
 \end{enumerate}
\end{prop}
For the general $\beta$ such that neither $\beta$, nor $\beta^{-1}$ is
single-valued one can describe the precise amount of strong convergence
in terms of the support of the Young measures associated to the weakly
convergent in $L^1(Q_T)$ sequences $(u_n)_n$, $(v_n)_n$.
We defer to \S\ref{sec:continuous} the discussion of this general setting,  the proof of Proposition~\ref{prop:Young} and extensions of this result (see Remark~\ref{rem:extensions-beta}).

\begin{rem}\label{rmk:exploit-bis}
There are other ways to exploit the property~\eqref{eq:compcomp}.
The way we propose in Proposition~\ref{prop:Young}
only brings $L^1$ convergence, which is not optimal if additional properties of $\beta$ are assumed
(see, e.g., the exploitation proposed in~\cite[\S4]{Moussa14} or~\cite[\S5]{DE15} based on the
convergence of some norm). However, the $L^1$ convergence is the crucial fact, and it can be 
upgraded using equi-integrability bounds for stronger $L^p$ norms. As a matter of fact, our result 
can be applied to a wide class of problems including strongly degenerate elliptic-parabolic
and parabolic-hyperbolic problems.
\end{rem}

Proposition~\ref{prop:Young} is proved in \S\ref{sec:continuous}, while \S\ref{sec:discrete} is devoted
to the extension to the discrete setting of Proposition~\ref{th:Moussa-compcomp}. A particular
attention is paid in~\S\ref{sec:discrete} to multistep discrete time differentiation like, e.g.,
Backward Differentiation Formulas (BDF). Finally, we apply our framework in \S\ref{sec:Porous}
for proving the convergence of two-point flux in space and BDF2 in time Finite Volume
approximation of the porous medium equation
$$
\p_t u - \Delta u^q = 0, \qquad q >1.
$$
But first, we give in~\S\ref{ssec:Richards} an example of how to use the combination of
Propositions~\ref{th:Moussa-compcomp} and~\ref{prop:Young} in the continuous setting in the case
of the so-called Richards equation modeling the unsaturated flow of water within a porous medium.
All the arguments we use can be transposed to the discrete setting, for example by using
the numerical method proposed in~\cite{CG_VAGNL}.

\subsection{A continuous example: Richards equation}\label{ssec:Richards}

We consider the Richards equation~\cite{Ric31,Bear72}
\be\label{eq:Richards}
\omega \p_t s - \div\left(\frac{k(s)}{\mu} \bbK (\grad p - \rho \g) \right) = 0 \quad \text{ in } Q_T,
\ee
where the two unknowns, namely the saturation $s \in L^\infty(Q_T;[0,1])$ and the pressure head $p$,
are linked by the capillary pressure relation $s = S(p)$ where $S:\R \to [0,1]$ is a nondecreasing function
that satisfies $S(p) = 0$ if $p\le 0$, $S(p)>0$ if $p>0$, and $(1-S)$ belongs to $L^1(\R_+)$.
The porosity $\omega \in L^\infty(Q_T)$ satisfies $\un \omega \le \omega \le \ov \omega$
a.e. in $Q_T$. The intrinsic permeability field $\bbK: \O \to \Mm_d(\R)$ satisfies $\bbK(\x) = \bbK(\x)^T$ and
$$
\un \k|\u|^2 \le \bbK(\x)\u \cdot \u \le \ov \k |\u|^2, \qquad \forall \u \in \R^d, \; \text{ for a.e. } \x \in \O
$$
for some $\un \k, \ov \k >0$. The relative permeability $k:[0,1] \to [0,1]$ is increasing and satisfies $k(0)= 0$,
leading to a (weak) degeneracy at $s=0$. The density $\rho$ and the viscosity $\mu$ are supposed to be constant,
and $\g$ denotes the gravity vector.
The equation~\eqref{eq:Richards} is complemented by the initial data
\be\label{eq:Richards-init}
s_{|_{t=0}} = s_0 \in L^\infty(\O;[0,1]),
\ee
and the Dirichlet boundary condition
\be\label{eq:Richards-bound}
p_{|_{\x \in \p\O}} = p_D \in H^1(\O; \R_+).
\ee
For discussions on more complex boundary conditions, see~\cite{Sch07,BKS11}.

In order to define properly the solution, we introduce the increasing Lipschitz continuous one-to-one mapping
$$
\phi:\begin{cases}
\R_+ \to \R_+ \\
p \mapsto \int_0^p \sqrt{k(S(a))}\d a,
\end{cases}
$$
that is extended to the whole $\R$ as an odd function.
\begin{Def}
\label{Def:Richards}
A couple $(s,p)$ is said to be a weak solution to~\eqref{eq:Richards}--\eqref{eq:Richards-bound}
if $\phi(p) - \phi(p_D) \in L^2((0,T);H^1_0(\O))$, $s = S(p) \in C([0,T]; L^1(\O))$ with $s_{|_{t=0}}= s_0$, and
\be\label{eq:Richards-Kirch}
\omega \p_t s - \div\left(\frac{\sqrt{k(s)}}{\mu} \bbK \left(\grad \phi(p) - \sqrt{k(s)}\rho \g \right) \right) = 0 \quad \text{ in } \Dd'(Q_T).
\ee
\end{Def}

We refer to~\cite{CG_VAGNL} for a recent result of convergence of a carefully designed entropy-consistent nonlinear finite volume scheme for \eqref{eq:Richards-Kirch} in the discrete compactness framework developed in the sequel. Although our main interest is to describe a tool for numerical analysis of degenerate parabolic problems, here we limit ourselves to the continuous framework. Indeed, first, proofs of convergence of numerical schemes are very similar in their spirit to the proofs of stability of solutions with respect to perturbation of data, coefficients or non-linearities. Second, the results of Propositions~\ref{th:Moussa-compcomp} and~\ref{prop:Young} are interesting already in the continuous framework. Therefore, here we focus on illustrating the structural stability feature (cf. e.g. \cite{A.BendahmaneKarlsenOuaroJDE09}) of the continuous problem
\eqref{eq:Richards}--\eqref{eq:Richards-bound}. Let us show that
a sequence $(s_n, p_n)_{n}$ of solutions (in the sense of Definition~\ref{Def:Richards})
to equations
\be\label{eq:Richards-n}
\omega_n \p_t s - \div\left(\frac{k(s)}{\mu} \bbK (\grad p - \rho \g) \right) = 0 \quad \text{ in } Q_T,
\ee
(playing the role of approximation of~\eqref{eq:Richards})
with the initial and boundary data~\eqref{eq:Richards-init}--\eqref{eq:Richards-bound}
 converges towards a solution of~\eqref{eq:Richards}--\eqref{eq:Richards-bound}
in the sense of Definition~\ref{Def:Richards}.
The only difference between \eqref{eq:Richards} and \eqref{eq:Richards-n} resides in introduction of approximate porosities $(\omega_n)_{n\ge1}$ that are supposed to satisfy properties~\eqref{eq:poids-conv}.

\medskip
The derivation of the \emph{a priori} bounds we propose here is formal. We refer to~\cite{CP12} for instance
for a rigorous derivation on a closely related problem.

\medskip
Since $s_n = S(p_n)$ with $S(\R)=[0,1]$, we obtain directly that $0 \le s_n \le 1$  a.e. in $Q_T$,
ensuring the $L^2$-weak convergence of an unlabeled subsequence $(s_n)_n$
towards some function $s \in L^\infty(Q_T;[0,1])$.

Multiplying (formally) the equation~\eqref{eq:Richards-n} by $p_n - p_D$ and integrating on $Q_T$
yields the estimate
$$
\iint_{Q_T} k(S(p_n)) |\grad p_n|^2 \d\x\d t =\iint_{Q_T} |\grad \phi(p_n)|^2 \d\x\d t \le C, \qquad \forall n \ge 1,
$$
for some $C$ independent on $n$. Thanks to Poincar\'e's inequality, one gets that
$$
\|\phi(p_n)\|_{L^2((0,T);H^1(\O))} \le C, \qquad \forall n \ge 1.
$$
In particular, there exists $\xi \in L^2((0,T);H^1(\O))$ with $\xi - \phi(p_D) \in  L^2((0,T);H^1_0(\O))$
such that $\phi(p_n)$
converges weakly in $L^2((0,T);H^1(\O))$ towards $\xi$ as $n\to\infty$.
It follows from  equation~\eqref{eq:Richards-Kirch} that
$$
\|\omega_n \p_t s_n\|_{L^2\left((0,T);H^{-1}(\O)\right)} \le C, \qquad \forall n \ge 1.
$$
Therefore, one can apply Proposition~\ref{th:Moussa-compcomp} to claim that
\be\label{eq:Richards-compcomp}
\omega_n s_n \phi(p_n) 
\rightarrow
\omega s\xi \quad \text{ in } \Dd'(Q_T) \text{ as } n\to \infty.
\ee

The sequences $\left(s_n\right)_n$ and $\left(\phi(p_n)\right)_n$ satisfy $\phi(p_n) \in \beta(s_n)$ for all $n\ge1$,
where $\beta$ is the maximal monotone graph with single valued inverse $\beta^{-1} = S\circ \phi^{-1}$.
We can now apply Proposition~\ref{prop:Young}. This first ensures that $\xi \in \beta(s)$, or equivalently,
defining $p:Q_T \to \R$ by $p : \phi^{-1}(\xi)$, that $s = S(p)$. Second, the almost everywhere convergence
of $s_n$ towards $s$ is ensured (up to a subsequence), so that one has
$$
s_n \underset{n\to\infty}\longrightarrow s \quad
\text{ and } \quad
k(s_n) \underset{n\to\infty}\longrightarrow k(s) \quad
 \text{ strongly in } L^r(Q_T), \quad \forall r \in [1,\infty).
$$
This is enough to pass to the limit in~\eqref{eq:Richards-n} and to claim that $(s,p)$ is a solution
to~\eqref{eq:Richards}--\eqref{eq:Richards-bound} in the sense of Definition~\ref{Def:Richards}.

\section{Proof of Proposition~\ref{prop:Young}}
\label{sec:continuous}

In order to avoid double integrals w.r.t. time and space, the functions we consider in this section
are defined on an open $\Oo$  subset of $\R^N$. $\Oo$ will play the role that $Q_T$ played
in the statement of Proposition~\ref{prop:Young}. This allows in particular to write $\bx \in \Oo$ rather that $(\x,t) \in Q_T$.
The weights $\omega_n$ and $\omega$ are extended to the whole $\Oo$, with
$0 < \un \omega \le \omega_n(\bx) \le \ov \omega$
and $\omega_n(\bx) \to \omega(\bx)$ for almost all $\bx \in \Oo$.

\begin{lem}\label{lem:beta-lim}
Let $\left(u_n\right)_n$ and $\left(v_n\right)_n$ be two sequences weakly converging
in $L^1(\Oo)$ towards $u$ and $v$ respectively
$v_n \in \beta(u_n)$ almost everywhere in $\Oo$ for all $n\ge 1$.
We suppose that $uv \in L^1(\Oo)$, that  $u_n v_n  \in L^1(\Oo)$ for all $n \ge 1$,
and that
$$
\int_\Oo \omega_n u_n v_n \varphi\,  \d\bx \underset{n\to\infty}{\longrightarrow} 
\int_\Oo \omega uv \varphi \, \d\bx, \qquad \forall \varphi \in \Dd(\O).
$$
Then $v \in \beta(u)$ almost everywhere in $\Oo$.
\end{lem}
\begin{proof}
The proof relies on the fundamental property of maximal monotone graphs:
\begin{multline}\label{eq:graph-max}
u \in {\rm dom}(\beta) \text{ and } v \in \beta(u) \\
\Leftrightarrow \qquad
(u-k)(v-K) \ge 0, \; \forall k \in {\rm dom}(\beta) \text{ and } K \in \beta(k).
\end{multline}
Let $k \in {\rm dom}(\beta)$ and $K \in \beta(k)$, then, for all  $\varphi \in \Dd(\Oo)$ such that $\varphi\ge 0$,
one has
 $$
  0\leq \iint_{Q_T} \omega_n (u_n-k)(v_n-K)\varphi\,\d\x, \qquad \forall n \ge 1.
 $$
Due to the assumptions of the lemma, passing to the limit $n \to \infty$ one finds
  $$
  0\leq \iint_{Q_T} \omega (u-k)(v-K)\varphi\,\d\x, \qquad \forall \varphi \in \Dd(\Oo) \text{ with } \varphi \ge 0.
 $$
The above property is sufficient to claim that $(u-k)(v-K)\ge 0$ a.e. in $\Oo$. One concludes the
proof of Lemma~\ref{lem:beta-lim} thanks to~\eqref{eq:graph-max}.
\end{proof}

  Keeping the notations of Lemma~\ref{lem:beta-lim},
  we set $w_n=u_n+v_n$. From Proposition~\ref{prop:monograph} we infer that
  $A:=(\text{Id}+\beta^{-1})^{-1}$ and $B:=(\text{Id}+\beta)^{-1}$
  are non-decreasing Lipschitz functions from $\R$ to $\R$ satisfying furthermore $A+B=\text{Id}$, $A(0)=0=B(0)$ and one checks easily that  $v_n=A(w_n)$ and $u_n=B(w_n)$.
  In the same way, we set $w=v+u$. From the assumptions of Lemma~\ref{lem:beta-lim}, we have
  $u_n\rightharpoonup u$ and  $v_n\rightharpoonup v$  weakly
  in $L^1(\Oo)$, whence $w_n\rightharpoonup w$ for the same topology. Since $v_n=A(w_n)$ and $u_n=B(w_n)$, we may use the fundamental
  theorem on representation of weakly convergent in $L^1(\Oo)$ sequences by Young
  measures (see \cite{Ball,Hunger-NY}), to deduce the existence of a Young
  measure $(\nu_{\bx}(\cdot))_{\bx\in \Oo}\subset \text{Prob}(\R)$ (here $\text{Prob}(\R)$
  is the class of all probability measures on $\R$) such that for a.e. $\bx\in \Oo$ there holds
  $$
  w(\bx)=\int_\R \lambda \,d\nu_{\bx}(\lambda),
  $$
  and
  \begin{equation}\label{eq:represent-v-and-u}
   v(\bx)=\int_\R A(\lambda) \,d\nu_{\bx}(\lambda), \qquad  u(\bx)=\int_\R B(\lambda) \,d\nu_{\bx}(\lambda).
  \end{equation}

  \begin{lem}\label{lem:Young2}
Let $\left(u_n\right)_{n}$ and $\left(v_n\right)_{n}$ be two sequences as in Lemma~\ref{lem:beta-lim}.
  For almost all $\bx \in \Oo$, we have either $A(\lambda)=v(\bx)$ for $\nu_{\bx}-$\,a.e. $\lambda\in\R$ or $B(\lambda)=u(\bx)$ for $\nu_{\bx}-$\,a.e. $\lambda\in\R$.
  \end{lem}
  \begin{proof}
  For $\ell \in \N$, we denote by $T_\ell:\R_+ \to \R$ the truncation function defined by
  $T_\ell(r) = \min\{r,\ell\}$. Observe that $T_\ell(A(w_n)B(w_n))\leq u_n v_n$ due to the fact that $AB\geq 0$. The function $w \mapsto T_\ell(A(w)B(w))$ being continuous
  and bounded, we can apply the fundamental theorem of \cite{Ball} and claim that
  for all $\ell \in \N$, one has
  \begin{multline*}
  \int_\Oo \omega uv\varphi\, \d\bx  = \lim_{n\to \infty}\int_\Oo \omega_n u_n v_n \varphi \d\bx
  \ge \lim_{n\to\infty}  \int_\Oo \omega_n T_\ell(u_n v_n)\varphi\d\bx
  \\=   \lim_{n\to\infty} \int_\Oo \omega_n T_\ell(A(w_n)B(w_n))\varphi\d\bx
  = \int_\Oo \omega \left(\int_{\R} T_\ell (A(\lambda) B(\lambda)) \d\nu_\bx(\lambda)\right)\varphi \d\bx.
  \end{multline*}
  Since this inequality holds for all $\ell \in \N$, it also holds for the limit $\ell \to \infty$
  that we can identify thanks to the monotone convergence theorem. Bearing in mind the
  representation~\eqref{eq:represent-v-and-u} of the functions $u$ and $v$, this ensures that,
  for all $\varphi \in \Dd(\Oo)$  with  $\varphi \ge 0$, there holds
   $$
   \int_\Oo   \omega \left(\int_\R A(\lambda)\d\nu_\bx(\lambda)\right) \left(\int_\R B(\lambda) \d\nu_\bx(\lambda)\right) \varphi \d\bx
   \ge \int_\Oo \omega \left(\int_\R A(\lambda) B(\lambda) \d\nu_\bx(\lambda)\right)\varphi \d\bx,
   $$
   or equivalently (see, e.g., \cite{Hunger-Duke}) that
    $$
   \int_{\Oo} \omega \varphi \left(\int_\R\!\int_\R \Bigl(A(\lambda)-A(\mu) \Bigr)\Bigl(B(\lambda)-B(\mu)\Bigr) \,d\nu_{\bx}(\lambda) \,d\nu_{\bx}(\mu)\right)\d\bx   \leq 0.
  $$
  The integrand of the above integral being nonnegative thanks to the monotonicity of $A$ and $B$ and the
  nonnegativity of $\varphi$ and $\omega$, it equals $0$. Since $\varphi$ is arbitrary, we deduce that for a.e. $\bx \in \Oo$,
  one has
  $$
  \int_\R\!\int_\R \Bigl(A(\lambda)-A(\mu) \Bigr)\Bigl(B(\lambda)-B(\mu)\Bigr) \,d\nu_{\bx}(\lambda) \,d\nu_{\bx}(\mu) = 0,
  $$
  whence $(A(\lambda)-A(\mu))(B(\lambda)-B(\mu))=0$ for $\nu_{\bx}\otimes\nu_{\bx}-$ a.e. $(\lambda,\mu)\in\R^2$.
   Hence it is not difficult to conclude that either $A$ is constant $\nu_{\bx}-$ a.e., or $B$ is is constant $\nu_{\bx}-$ a.e., the value of the corresponding constant being fixed by \eqref{eq:represent-v-and-u}. This concludes the proof of Lemma~\ref{lem:Young2}.
  \end{proof}

Let us introduce the sets $\Uu$ and $\Vv$ defined by
 \begin{align*}
  \Uu =& \{ \bx \in \Oo \; | \; B = u(\bx) \text{ } \nu_\bx-\text{a.e.}  \}, \\
  \Vv =& \{ \bx \in \Oo \; | \; A = v(\bx) \text{ } \nu_\bx-\text{a.e.} \}.
 \end{align*}
  It follows from Lemma~\ref{lem:Young2} that $\Oo \setminus (\Uu \cup \Vv)$ is negligible.
  \begin{lem}\label{lem:Young3}
 Under the assumptions of Lemma~\ref{lem:beta-lim}, one has
 $$
 u_n \underset{n\to\infty}\longrightarrow u \text{ strongly in } L^1(\Uu) \quad \text{ and } \quad
 v_n \underset{n\to\infty}\longrightarrow v \text{ strongly in } L^1(\Vv).
 $$
   \end{lem}
  \begin{proof}
  We will prove that $u_n\to u$ strongly in $L^1(\Uu)$, the proof of $v_n \to v$ being similar.
 Since $u \in L^1(\Uu)$, it can be approximated by simple functions: for all $\eps >0$, there exist
 an integer $I_\eps$ and a simple
 function $u^\eps = \sum_{i=1}^{I_\eps} \k_i \1_{E_i}$, where each $\k_i$ ($1\le i \le I_\eps$) is a
 real value and each  $E_i$ is a measurable subset of $\Uu$ with $\bigcup_i E_i = \Uu$, such that
 \be\label{eq:u-u^eps}
 \|u^\eps-u\|_{L^1(\Uu)} \le \eps.
 \ee
 Without loss of generality, we can assume that $E_i$ is bounded if $\k_i \neq 0$.
 For all $i \in \{1,\dots, I_\eps\}$, the function $w \mapsto |B(w)-\k_i|$ is Lipschitz continuous, ensuring
 the uniform equi-integrability of the sequence $\left(|B(w_n)-\k_i|\right)_{n}$, so
 that
 $$
 \int_\Uu |u_n - u^\eps|\d\bx = \sum_{i = 1}^{I_\eps} \int_{E_i} |B(w_n) - \k_i| \d\bx
 \underset{n\to\infty}{\longrightarrow}
  \sum_{i = 1}^{I_\eps} \int_{E_i} \int_\R |B(\lambda) - \k_i| \d\nu_\bx(\lambda) \d\bx.
 $$
 Since $E_i$ is a subset of $\Uu$, we have
$$
 \int_\Uu |u_n - u^\eps|\d\bx\underset{n\to\infty}{\longrightarrow} \int_\Uu |u-u^\eps|\d\bx, \qquad \forall \eps >0.
$$
 It follows from the triangle inequality and from~\eqref{eq:u-u^eps} that
 $$
 \underset{n\to \infty}{\rm limsup}  \int_\Uu |u_n - u|\d\bx \le 2 \|u^\eps-u\|_{L^1(\Uu)} \le 2 \eps, \qquad \forall \eps >0.
 $$
 Since $\eps$ is arbitrary, we find that $u_n \to u$ strongly in $L^1(\Uu)$.
  \end{proof}
	
  The last lemma of this section focuses on the case where either $\beta$ or $\beta^{-1}$
  is single-valued.
  \begin{lem}\label{lem:Young4}
 Under the assumptions of Lemma~\ref{lem:beta-lim}, and under the additional assumption that
 $\beta$ is single-valued,  $v_n$ converges to $v$ strongly in $L^1(\Oo)$ as $n \to \infty$. Similarly, if $\beta^{-1}$
 is single valued, then $u_n$ converges to $u$ strongly in $L^1(\Oo)$ as $n \to \infty$.
  \end{lem}
  \begin{proof}
  Assume that $\beta$ is single-valued, then the function $B$ is (strictly) increasing, thus it is non-constant on any non-trivial interval.
  Therefore, Lemma~\ref{lem:Young2} implies that, up to a negligible set, $\Oo = \Vv$, and one concludes by
  using Lemma~\ref{lem:Young3}.
    \end{proof}

\begin{rem}\label{rem:extensions-beta}
The assumption that maximal monotone graph $\beta$ satisfy $0 \in \beta(0)$ is easily dropped, indeed,
it is enough to change in Proposition~\ref{prop:Young} the functions $u_n,v_n$ into $u_n-k,v_n-K$ respectively with $(k,K)\in \beta$.
It is also immediate to extend the result of Proposition~\ref{prop:Young} to a measurable in $x$ family of maximal monotone graphs
$(\beta(\bx,\cdot))_{\bx\in \Oo}$ (to fix the ideas, we can enforce measurability by requiring that the functions $(\bx,z)\mapsto (\text{Id}+\beta(\bx,\cdot))^{-1}(z)$,
$(\bx,z)\mapsto (\text{Id}+\beta^{-1}(\bx,\cdot))^{-1}(z)$ be Carath\'eodory), and $0\in \beta(\bx,0)$ for a.e. $\bx\in\Oo$.
At a price of some simple additional assumptions on $(\beta_n)_n$, one can also consider the case of a sequence of convergent nonlinearities:
$v_n(\cdot)\in\beta_n(\cdot,u_n(\cdot))$.

Finally, observe that the assumptions of weak $L^1$ convergence in $\Oo$ can be turned into $L^1_{loc}(\Oo)$ weak convergence assumptions;
in this case, the conclusions in Lemmas~\ref{lem:Young3},\,\ref{lem:Young4} will turn into strong $L^1_{loc}$ convergences.
\end{rem}

\section{Discrete ``compensated compactness'' result and its consequences.}\label{sec:discrete}

Our goal is now to derive a discrete counterpart of Proposition~\ref{th:Moussa-compcomp},
namely Proposition~\ref{thm:compcomp-disc}.  Combined with Proposition~\ref{prop:Young}, it leads to
Theorem~\ref{thm:main} that can be used as a black-box.

One needs to define discrete operators for defining a discrete function, its gradient and its time derivative.
Rather than focusing on a particular numerical method, we concentrate on the fundamental
properties a numerical method has to fulfill so that our result holds. This motivates the use of the
so-called \emph{gradient scheme} framework~\cite{DEGH13} for the spatial discretization in \S\ref{ssec:space-discr}.
Concerning the time discretization, the approach discussed in \S\ref{ssec:time-discr} allows to consider either
some one-step discretization methods, like for instance Euler and Runge-Kutta methods, or multistep methods,
like Backward Differentiation Formula (BDF). The main result of this section, namely Theorem~\ref{thm:compcomp-disc},
is stated and proved in~\S\ref{ssec:discrcompcomp}.

In what follows, $\O$ is supposed to be a Lipschitz continuous open bounded subset of $\R^d$.
We restrict our attention to the case of cylindrical discretizations of
$Q_T:=\O\times(0,T)$, i.e., discretizations obtained thanks to a discretization of $\O$ and a discretization of $(0,T)$.

\subsection{Spatial discretization}\label{ssec:space-discr}

Concerning the space discretization,
in order to be able to consider a wide range of possible numerical methods
(including several Finite Elements with mass lumping and Finite Volume methods),
we stick to the \emph{Gradient Schemes} framework developed in~\cite{DEGH13}.

Let $m \in \N^\ast$ be the number of degrees of freedom in space, we assume that there exist two linear operators
$$  \pi_{m}: \R^m \to L^\infty(\O), \qquad \grad_m: \R^m \to L^\infty(\O)^d$$
such that the below assumptions (${\bf A}_\x$1)--(${\bf A}_\x$3) hold.

\bigskip
\begin{enumerate}[(${\bf A}_\x$1)]

\item\label{A:compact} 
For all $p \in [1,\infty)$ and all $m\ge 1$, there exists a norm $\u \mapsto \|\u\|_{p,m}$ on $\R^m$ such that
the following assumption on the space translates holds:
\be\label{eq:translates_x}
\lim_{|\boldsymbol\zeta|\to0} \sup_{m\ge 1} \sup_{\bv_m \in \R^m \setminus \{0\}}
\frac{{\|  \pi_m \bv_m(\cdot + \boldsymbol\zeta) -  \pi_m \bv_m \|}_{L^p(\O)}}{{\| \bv_m \|}_{p,m}} = 0, \quad \forall p \in [1,+\infty),
\ee
the function $ \pi_m v_m$ being extended by $0$ outside of $\O$.
\end{enumerate}

\noindent
 In particular, a bounded  w.r.t. the norm $\| \cdot \|_{p,m}$ sequence
$\left(\u_m^{(k)}\right)_{k \ge 0} \subset \R^m$ yields a relatively compact
sequence $\left( \pi_m \u_m^{(k)}\right)_{k \ge 0}$ in $L^p(\O)$.
\begin{rem}
A classical choice for the norm $\|\cdot\|_{p,m}$ is
$$
\|\u_m\|_{m,p} := \|  \pi_m \u_m\|_{L^p(\O)} + \|\grad_m \u_m\|_{L^p(\O)^d}, \qquad \forall \u_m \in \R^m,
$$
as suggested in~\cite{DEGH13}.
Some spatial discretizations enjoy discrete Sobolev
injections (see for instance the appendix of~\cite{EGH10} for the case of Dirichlet boundary conditions and Appendix B of \cite{ABRB11} for the Neumann ones), i.e.,  the property~\eqref{eq:translates_x} holds if one sets
$$
\|\u_m\|_{m,p} := \|  \pi_m \u_m\|_{L^q(\O)} + \|\grad_m \u_m\|_{L^q(\O)^d}, \qquad \forall \u_m \in \R^m
$$
for any $q > \frac{pd}{d+p}$. Notice that $q$ can be chosen strictly smaller than $p$ in that case, and that
on the contrary to the notation adopted in~\cite{DEGH13}, the subscript
$p$ in $\|\cdot\|_{p,m}$ does not necessary refer to $W^{1,p}(\O)$, but
  it refers to the more general
  property of ``uniform in $m$ compactness'' of the operator $\pi_m:(\R^m,\|\cdot\|_{m,p}) \to L^p(\Omega)$.
 \end{rem}

\medskip
\begin{enumerate}[(${\bf A}_\x$2)]
\item\label{A:beta-consistent} 
Let $\u_m = \left(u_{m,i}\right)_{1\le i \le m}$
and $\bv_m = \left(v_{m,i}\right)_{1\le i \le m}$ be two vectors of $\R^m$ such that $\bv_m \in \beta(\u_m)$, i.e.,
$v_{m,i}\in\beta(u_{m,i})$ for all $i \in \{1,\dots, m\}$, then $ \pi_m \u_m \in \beta( \pi_m \bv_m)$ for a.e. $\x \in \O$.
\end{enumerate}

\smallskip\noindent
We should stress that this assumption is restrictive: typically, it is fulfilled for piecewise constant discrete solutions produced by reconstruction operator $ \pi_m$, and therefore our analysis is suitable for finite volume methods and for the simplest finite element methods (in particular, methods with mass lumping).

\bigskip
Before formulating the last assumption, we need more notation. In the sequel, $\left(\omega_m\right)_{m\ge 1} \subset L^\infty(\O)$ denotes a sequence for which there exists
$\un \omega, \ov \omega >0$ such that
\be\label{eq:omega_m-bound}
 \un \omega \le \omega_m \le \ov \omega, \quad \text{ a.e. in } \O, \forall m \ge 1.
\ee
We additionally assume that there exists $\omega \in L^\infty(\O)$ such that
\be\label{eq:omega_m-conv}
\omega_m \underset{m\to\infty}\longrightarrow \omega \quad \text{ almost everywhere in $\O$}.
\ee
Obviously, one has
 $\un \omega \le \omega \le \ov \omega$ a.e. in  $\O$.

\medskip
Given $\varphi \in C^\infty_c(\O)$,
our last assumption is formulated in terms of the subset  $\Pp_m(\varphi)$ of $\R^m$ defined by
$$
\Pp_m(\varphi) = \left\{ \bv_m \in \R^m \; \left| \; \int_\O \omega_m  \pi_m \u_m ( \pi_m \bv_m - \varphi) \d\x = 0, \; \forall \u_m \in \R^m \right\}\right. .
$$

\bigskip
\begin{enumerate}
\item[$({\bf A}_\x3)$] There exists a linear operator $\boldsymbol{P}_m: C^\infty_c(\O) \to \R^m$ such that $\boldsymbol{P}_m\varphi \in \Pp_m(\varphi)$, and
such that there exists $C$ not depending on $m$ such that
$$
\| \grad_m \boldsymbol{P}_m \varphi \|_{L^\infty(\O)} \le C \| \grad \varphi \|_{L^\infty(\O)}, \qquad \forall \varphi \in C^\infty_c(\O).
$$
\end{enumerate}
\noindent The set $\Pp_m(\varphi)$ is the preimage under $ \pi_m$ of the $L^2_{\omega_m}(\O)$ projection of $\varphi$ on the range of $ \pi_m$.
It is not empty, and it reduces to a singleton in the particular case where $\bv_m \mapsto \| \pi_m \bv_m\|_{L^2(\O)}$ is a norm on $\R^m$.
In the latter case, the linearity of $\boldsymbol{P}_m$ is automatic since the range of $\pi_m$ is a finite-dimensional subspace of $L^2_{\omega_m}(\O)$.
Since $\boldsymbol{P}_m\varphi \in \Pp_m(\varphi)$, there holds
\be\label{eq:Pm-prop}
\int_\O \omega_m  \pi_m \u_m ( \pi_m \boldsymbol{P}_m \varphi - \varphi) \d\x = 0, \quad \forall \varphi \in C^\infty_c(\O), \forall u_m \in \R^m.
\ee

\medskip
\begin{rem}
Since we focus here only on compactness properties, we do not require the reconstruction
operators $ \pi_m$ and $\grad_m$ to fulfill the natural consistency and conformity relations  that are required in~\cite{DEGH13} for proving the convergence of the methods based on gradient schemes.
For instance, in problems that are posed in a Sobolev space $W^{1,p}$, the consistency and conformity relations read, respectively:
 $\forall \varphi \in W^{1,p}(\O)$,
\be\label{eq:consistence_x}
S_m(\varphi) := \min_{\bv_m \in \R^m} \left( \|  \pi_m \bv_m - \varphi \|_{L^p(\O)} + \| \grad_m \bv_m  - \grad \varphi \|_{L^p(\O)^d}\right)
\underset{m \to \infty}{\longrightarrow} 0;
\ee
and $\forall \boldsymbol \varphi \in \Cc^1(\ov \O; \R^d)$,
\be\label{eq:conformite_x}
W_m(\boldsymbol\varphi) := \max_{\bv_m \in \R^m} \frac{1}{{\| \bv_m \|}_{p,m}}\left| \int_{\O}( \grad_m \bv_m \cdot \boldsymbol\varphi +  \pi_m \bv_m \div \boldsymbol\varphi ) \d\x \right|
\underset{m \to \infty}{\longrightarrow} 0.
\ee
\end{rem}

\subsection{Time discretization}\label{ssec:time-discr}

\subsubsection{Extension of the spatial reconstruction operators} \label{sssec:time-space}

A time discretization of $(0,T)$ consists in a subdivision $0 = t_0 < t_1 < \dots < t_n = T$ of the interval $[0,T]$.
For $k \in \{1,\dots, n\},$ we denote $\dt_{n,k} = t_k - t_{k-1}$ and $\dt_n = \max_{1 \le k \le n} \dt_{n,k}$.

\smallskip
This allows to extend the operators $ \pi_m: \R^m \to L^\infty(\O)$ and
$\grad_m:\R^m \to \left(L^\infty(\O)\right)^d$ into
$$
 \pi_{m}^n: \R^{m\times (n+1)} \to L^\infty (Q_T) \quad 
 \text{and}\quad \grad_{m}^n:\R^{m\times (n+1)} \to \left(L^\infty
 (Q_T)
 \right)^d
$$
as follows. Given $\u_m^n = \left( u_{m,i}^{n,k} \right)_{1 \le i \le m}^{0 \le k\le n}
\in \R^{m\times (n+1)}$, we set $\u_m^{n,k} = \left( u_{m,i}^{n,k} \right)_{1 \le i \le m} \in \R^m$ for all $k \in \{0,\dots, n\}$.
Then for $k \in \{1,\dots, n\}$, we set
\begin{subequations}\label{eq:Pi_m^n}
\be\label{eq:Pi_m^nk}
 \pi_m^n \u_m^n (\cdot,t) =  \pi_m \u_m^{n,k} \quad \text{and} \quad \grad_m^n \u_m^n (\cdot,t) = \grad_m \u_m^{n,k}
\quad \text{ if } t \in (t_{k-1},t_k],
\ee
and (formally, since $\{0\}$ is of zero measure in $[0,T]$)
\be\label{eq:Pi_m^n0}
 \pi_m^n \u_m^n(\cdot,0) =  \pi_m \u_m^{n,0} \quad \text{and} \quad \grad_m^n \u_m^n (\cdot,0) = \grad_m \u_m^{n,0}.
\ee
\end{subequations}
Thanks to these definitions, we can define a semi-norm $\| \cdot \|_{p,m,q,n}$ on $\R^{m\times(n+1)}$ by
\be\label{eq:norm-xpmn}
\| \u_m^n \|_{p,m,q,n}= \left( \sum_{k=1}^n \dt_{n,k} {\| \u_m^{n,k}\|}^q_{p,m} \right)^{1/q}.
\ee
Let us now extend the operator $\boldsymbol{P}_m$ to the case of time dependent functions. We introduce the linear operator
$\boldsymbol{P}_m^n : C([0,T];
C^\infty_c(\O)) \to \R^{m\times(n+1)}$ defined by
\begin{subequations}\label{eq:Pmn}
\be\label{eq:Pmnk}
\left(\boldsymbol{P}_m^n \varphi\right)^k = \boldsymbol{P}_m \varphi(\cdot, t_{k-1}), \quad \forall k \in \{1,\dots , n\}, \; \forall \varphi \in C([0,T];
C^\infty_c(\O)),
\ee
\be\label{eq:Pmn0}
\left(\boldsymbol{P}_m^n \varphi\right)^0 = \boldsymbol{P}_m \varphi(\cdot, 0), \quad \forall \varphi \in C([0,T];L^2(\O)).
\ee
\end{subequations}
It results from Assumption~$({\bf A}_\x3)$ that there exists $C>0$ depending neither on $m$ nor on $n$ such that
\be\label{eq:control-gradmn}
\| \grad_m^n \boldsymbol{P}_m^n \varphi \|_{L^\infty(Q_T)} \le C \|\grad \varphi \|_{L^\infty(Q_T)}, \qquad \forall \varphi \in C^\infty_c(Q_T).
\ee

\subsubsection{The one-step discrete differentiation operator}\label{sssec:one-step}

It remains to define a reconstruction operator $\delta_m^n$ in order to approximate the time
derivative of the function $ \pi_m^n \u_m^n$.
Since the case of the one-step time differentiation operator plays a fundamental role in the analysis carried out in this paper,
we first define
$\delta_m^n : \R^{m\times(n+1)} \to L^\infty
(Q_T)
$ by
\be\label{eq:delta_m^n}
\delta_m^n \u_m^n (\cdot,t) = \frac{ \pi_m \u_m^{n,k} -  \pi_m \u_m^{n,k-1}}{\dt_{n,k}} \quad \text{ if } t \in (t_{k-1},t_k], \; \forall \u_m^n \in \R^{m\times(n+1)}.
\ee

The one-step time differentiation operator $\delta_m^n$ enjoys the following particular consistency property.
\begin{lem}\label{lem:constability}
For all $\varphi \in W^{1,1}((0,T);C^\infty_c(\O))$
such that $\varphi(\cdot,0) = \varphi(\cdot,T) = 0$, and for all
$\u_m^n \in \R^{m\times(n+1)}$, one has
$$
 \iint_{Q_T} \left(\omega_m \delta_m^n \u_m^n  \pi_m^n \boldsymbol{P}_m^n \varphi + \omega_m \pi_m^n \u_m^n \p_t \varphi \right)
 \d\x \d t =0.
$$
\end{lem}
\begin{Proof}
Thanks to the definition~\eqref{eq:Pi_m^n} of the reconstruction operator $ \pi_m^n$, using the classical summation-by-parts procedure one has
\begin{align*}
\iint_{Q_T}\omega_m  \pi_m^n \u_m^n \p_t \varphi \d\x \d t
=& \sum_{k=1}^n  \int_\O \omega_m  \pi_m \u_m^{n,k} \left(\int_{t_{k-1}}^{t^k} \p_t \varphi \,\d t\right) \d \x \\
=&   \sum_{k=1}^n  \int_\O \omega_m  \pi_m \u_m^{n,k} \left(\varphi(\cdot,t_k) - \varphi(\cdot,t_{k-1}) \right) \d\x \\
= & \sum_{k=1}^n \int_\O \omega_m \varphi(\cdot,t_{k-1})   \pi_m\left( \u_m^{n,k-1}-  \u_m^{n,k}\right) \d \x.
\end{align*}
Thanks to~\eqref{eq:Pm-prop}, we obtain that
\begin{multline*}
\iint_{Q_T} \omega_m  \pi_m^n \u_m^n \p_t \varphi \d\x \d t \\
=
\sum_{k=1}^n \int_\O \omega_m  \pi_m \boldsymbol{P}_m \varphi(\cdot,t_{k-1}) \left(  \pi_m \u_m^{n,k-1}(\x) -  \pi_m \u_m^{n,k}(\x)\right) \d \x.
\end{multline*}
The result of Lemma~\ref{lem:constability} stems from the definition~\eqref{eq:Pmn} of
the operator $\boldsymbol{P}_m^n$ and from the definition~\eqref{eq:delta_m^n} of the operator $\delta_m^n$.
\end{Proof}
\medskip

\subsubsection{From one-step to multi-step operators}\label{sssec:multi-step}
The key idea of our reduction argument proposed below is to represent multi-step differentiation operators as linear combinations of Euler backward differences with shifted time: for instance, $$\frac 32 u^k - 2 u^{k-1} + \frac 12 u^{k-2} = \frac 32 (u^k-u^{k-1})-\frac 12(u^{k-1}-u^{k-2}).$$ We introduce the appropriate matrix formalism for computations based on this idea; it will be exploited in Lemma~\ref{lem:multi-to-semi} and in Proposition~\ref{prop:convergence}.
Let us represent $\u_m^n \in \R^{m\times(n+1)}$ by the matrix
\be\label{eq:umn-mat}
\u_m^n = \begin{pmatrix}
u_{m,1}^{n,0} & \dots & u_{m,m}^{n,0} \\
\vdots & & \vdots \\
u_{m,1}^{n,n} & \dots & u_{m,m}^{n,n}
\end{pmatrix} \in \Mm_{n+1,m}(\R).
\ee
Then the operator $\delta_m^n$ corresponding to one-step discretization introduced in~\eqref{eq:delta_m^n}
can be rewritten as $\delta_m^n  =  \pi_m^n \circ \M_n$, where we have set
\be\label{eq:Mn}
\M_n = \begin{pmatrix}
1 & 0 & \dots & \dots &\dots & 0 \\
-\frac{1}{\dt_{n,1}} & \frac{1}{\dt_{n,1}} & 0 & \dots & \dots &0 \\
0 & -\frac{1}{\dt_{n,2}} & \frac{1}{\dt_{n,2}} & \ddots  &  & \vdots \\
\vdots & \ddots& \ddots & \ddots & \ddots & \vdots \\
\vdots &  & \ddots & \ddots & \ddots & 0 \\
0 & \dots & \dots & 0 & - \frac{1}{\dt_{n,n}} & \frac{1}{\dt_{n,n}}
\end{pmatrix} \in \Mm_{n+1}(\R).
\ee
\begin{rem}\label{rem:1st-line}
Since $\u_m^{n,0}$ is only used on the negligible set $\{t=0\}$ of $[0,T]$ in the reconstruction $ \pi_m^n \u_m^n$
defined in~\eqref{eq:Pi_m^n}, the choice of the first line of the matrix $\M_n$ is arbitrary. The particular choice
we did in~\eqref{eq:Mn} leads to a matrix $\M_n$ that is lower triangular and invertible.
\end{rem}
Denoting by $\D_n$ and $\T_n$ the matrices of $\Mm_{n+1}(\R)$ defined by
$$
\D_n = \begin{pmatrix}
1 & 0 & \cdots & \cdots & 0 \\
-1 & 1 & \ddots & & \vdots\\
0 & \ddots & \ddots & \ddots & \vdots \\
\vdots & \ddots & \ddots & 1 & 0 \\
0 & \cdots & 0 & -1 & 1
\end{pmatrix},
\qquad
\T_n = \begin{pmatrix}
1 & 0 & \cdots & 0 \\
0 & \dt_1 & \ddots & \vdots \\
\vdots & \ddots & \ddots & 0 \\
0 & \cdots & 0 & \dt_n
\end{pmatrix},
$$
we get that $\M_n = \T_n^{-1} \D_n$.
\medskip

In what follows, we restrict our study to the discrete time differentiation operators $\h\delta_m^n$ that can be defined by
\be\label{eq:hdmn0}
\h\delta_m^n \u_m^n  =  \pi_m^n \circ \h\M_n \u_m^n, \quad \forall \u_m^n \in \R^{m\times(n+1)}
\ee
for some lower triangular invertible matrix $\h \M_n$ belonging to $\Mm_{n+1}(\R)$. Following the discussion
of Remark~\ref{rem:1st-line}, we can enforce the first line of $\h\M_n$ --- denoted with subscript $0$ in
accordance with~\eqref{eq:umn-mat} --- to be equal to the first line of $\M_n$, namely
\be\label{eq:hMn0}
\left(\h\M_n\right)_{0,0} = 1, \quad \left(\h\M_n\right)_{0,k} = 0 \text{ if } k \in \{1,\dots,n\}.
\ee
Requiring that $\h\M_n$ is a lower triangular matrix means that for approximating $\p_t u$ on $(t_{k-1},t_k]$,
one can only use the vectors $\left(\u_m^{n,\ell}\right)_{0 \le \ell\le k}\subset \R^m$, which is fairly natural.
Requiring that $\h\M_n$ is invertible means that, knowing the initial value
$ \pi_m^n \u_m^n(\cdot,0) =  \pi_m(\u_m^{n,0})$ of $ \pi_m^n \u_m^n$ and
its approximate time derivative $\h\delta_m^n \u_m^n$, one can reconstruct $ \pi_m^n \u_m^n$.
\medskip

We require a last very natural property on $\h\delta_m^n$, that is supposed to vanish on constant w.r.t. time vectors
of discrete unknowns.
More precisely, let $\u_m^n = \left(\u_m^{n,k}\right)_{0 \le k \le n} \in \R^{m\times(n+1)}$ be such that
$\u_m^{n,k}=\u_m^{n,k-1}$ for all $k \in \{1,\dots,n\}$, then $\h\delta_m^n \u_m^n = 0$ a.e. in $Q_T$.
This amounts to assuming that
\be\label{eq:sum-line}
\sum_{\ell = 0}^n \left(\h\M_n\right)_{k,\ell} = 0, \qquad \forall k \in \{1,\dots,n\}.
\ee

\begin{lem}\label{lem:An}
Under the above assumptions, there exists a unique invertible lower triangular matrix $\A_n\in \Mm_n(\R)$ such that
\be\label{eq:An}
\h\A_n := \T_n \h\M_n \D_n^{-1} = \begin{pmatrix}
1 & 0 \; \cdots \; 0 \\
\begin{array}{c}
0 \\ \vdots \\ 0
\end{array} &
\A_n
\end{pmatrix}.
\ee
\end{lem}
\begin{proof}
The fact that $\A_n$ is lower triangular and invertible follows directly from the fact that $\T_n$, $\h\M_n$ and $\D_n^{-1}$ are.
The only thing to be checked is that the first column of $\h\A_n$ is equal to $(1,0,\dots, 0)^T$.
It is first easy to check that
$$
\D_n^{-1} = \begin{pmatrix}
1 & 0 & \cdots & 0 \\
1 & \ddots & \ddots & \vdots \\
\vdots & \ddots & \ddots & 0 \\
1 & \cdots & 1 & 1
\end{pmatrix}
.
$$
The property $\left(\h\A_n\right)_{0,0} = 1$ follows from the particular choice~\eqref{eq:hMn0} of the first line of $\h\M_n$,
while the property $\left(\h\A_n\right)_{k,0} = 0$ for $k \ge 1$ follows from~\eqref{eq:sum-line}.
\end{proof}

\medskip

With the above formalism, we thus focus on discrete time-differentiation operators of the form
\be\label{eq:hdmn}
\h\delta_m^n \u_m^n =  \pi_m^n \left( \T_n^{-1} \h \A_n \T_n \M_n \u_m^n\right), \qquad \forall \u_m^n \in \R^{m\times(n+1)},
\ee
with $\h\A_n$ of the form~\eqref{eq:An}.
Note that the one-step differentiation enters this framework. In this case, the matrix $\A_n$ reduces to the identity.
\begin{rem}\label{rem:uniform-time}
In the particular case of a uniform time discretization, i.e., $\dt_{n,k} = T/N=\dt_n$ for all $k\in \{1,\dots,n\}$, it results from the
particular structure of the matrix $\h\A_n$ (see Lemma~\ref{lem:An}) that $\h\A_n$ and $\T_n$ commute, so that
$\h\M_n = \h\A_n \M_n$.
\end{rem}

\smallskip
The assumption we make on the discrete time-differentiation operator is:
\begin{enumerate}[$({\bf A}_t)$]
\item\label{At} Let $\h \delta_m^n: \R^{m\times(n+1)} \to L^\infty(\O)$ be an operator of the form~\eqref{eq:hdmn}.
where $\h\A_n \in \Mm_{n+1}(\R)$ and $\A_n \in \Mm_n(\R)$  are the matrix defined by~\eqref{eq:An}. We assume that
there exists $C$ not depending on $n$ such that ${\| \A_n^{-1} \|}_1 \le C$, where $\|\cdot\|_1$ denotes the
usual matrix $1$-norm, i.e.,
${\| \mathbb B \|}_1 = \max_{1 \le j \le n} \sum_{i=1}^n | \mathbb B_{i,j} |$ for all  $\mathbb B \in \Mm_n(\R).$

\end{enumerate}
\noindent Let us stress that this assumption is fulfilled by the one-step differentiation operators with arbitrary time steps
since $\A_n$ reduces to Identity. However, for multi-step methods like BDF2 described below, one may need to constraint the ratio between adjacent time intervals in order to guarantee that $({\bf A}_t)$ holds.

\smallskip
\begin{rem}\label{rem:BDF2}
Let us illustrate (in the particular case of a uniform discretization) how to determine the matrix $\A_n$ prescribed
by~\eqref{eq:An} for the so-called BDF2 scheme (see e.g.~\cite{SM03}).
The principle a such a method consists in an initialization with a one-step differentiation
$$
\h \delta_m^n \u_m^n(\cdot,t) =  \pi_m\left( \frac{\u_m^{n,1}-\u_m^{n,0}}{\dt_n}\right), \quad
\text{ if } t \in (0,\dt_n), \quad \forall \u_m^n \in \R^{m\times(n+1)},
$$
while $\h \delta_m^n$ is defined on the following time steps by: $\forall \u_m^n \in \R^{m\times(n+1)}$, $\forall k \in \{2,\dots ,n\}$,
$$
\h \delta_m^n \u_m^n(\cdot,t) =  \pi_m\left( \frac{\frac32 \u_m^{n,k}-2\u_m^{n,k-1} + \frac12 \u_m^{n,k-2}}{\dt_n}\right), \quad
\text{ if } t \in ((k-1)\dt_n,k\dt_n).
$$
This leads to
\begin{align*}
\h \delta_m^n \u_m^n(\x,t)=& \delta_m^n \u_m^n(\x,t) \1_{(0,\dt_n)}(t) \\
&+
\left(\frac{3}{2} \delta_m^n \u_m^n(\x,t) - \frac12  \delta_m^n \u_m^n(\x,t-\dt_n)\right) \1_{(\dt_n,T)}(t),
\end{align*}
and thus, in view of Remark~\ref{rem:uniform-time}, to
\be\label{eq:An-BDF2}
\A_n = \begin{pmatrix*}[r]
1 & 0 & \cdots & \cdots & 0 \\
-\frac12 & \frac32 & \ddots & & \vdots \\
0 & \ddots & \ddots & \ddots & \vdots \\
\vdots & \ddots & \ddots & \frac32 & 0\\
0 & \cdots & 0 & -\frac12 & \frac32
\end{pmatrix*}.
\ee
It is easy to verify that
$${ \| (\A_n)^{-1} \|}_1 = \frac32\left(1-\frac1{3^n}\right) \le \frac32, \qquad \forall n \ge 1,$$
so that the BDF2 method with uniform stepping satisfies Assumption~(${\bf A}_t$).
\end{rem}

\subsection{Discrete ``compensated compactness'' feature}\label{ssec:discrcompcomp}

We can now state the main result of this section, which is a discrete version of Proposition~\ref{th:Moussa-compcomp}.
\begin{prop}\label{thm:compcomp-disc}
Let $\left(\pi_m^n\right)_{m,n\ge1}$ and $\left(\grad_m^n\right)_{m,n\ge1}$ be discrete
reconstruction operators as defined in~\S\ref{sssec:time-space}.
We assume that Assumptions~$({\bf A}_\x1)$ and $({\bf A}_\x3)$ hold.
Let $\left(\h\delta_m^n\right)_{m,n\ge1}$ be a family of discrete time-differentiation operators
of the form~\eqref{eq:hdmn} satisfying Assumption~$({\bf A}_t)$.
Let $(\u_m^n)_{m,n\ge1}$ and $(\bv_m^n)_{m,n\ge1}$ be two families of vectors such that
$\u_m^n, \bv_m^n \in \R^{m\times(n+1)}$.
Let $\left(\omega_m\right)_{m\ge 1} \subset L^\infty(\O)$
be a sequence of functions such that~\eqref{eq:omega_m-bound} and~\eqref{eq:omega_m-conv} hold.
We assume that there exist $p$ and $q$ in $(1,\infty)$ and $C$ depending neither on $m$ nor on $n$
such that, for all $m,n \ge 1$,
\begin{enumerate}[{\bf a.}]
\item $\| \u_m^n \|_{p,m,q,n} \le C$, in particular  $\left( \pi_m^n\u_m^n\right)_{m,n\ge1}$ admits a weak limit (up to a subsequence)
$u$ in $L^q((0,T);L^p(\O))$ as $m,n\to \infty$;
\item $\|\pi_m^n \bv_m^n\|_{L^{q'}((0,T);L^{p'}(\O))} \le C$, so that $\left( \pi_m^n\bv_m^n\right)_{m,n\ge1}$ admits a
weak limit (up to a subsequence) $v$ in $L^{q'}((0,T);L^{p'}(\O))$ as $m,n\to \infty$;
\item for all $\bphi_m^n \in \R^{m\times(n+1)}$, one has
\be\label{eq:estim-hdmn}
\iint_{Q_T} \omega_m \h\delta_m^n \u_m^n  \pi_m^n \bphi_{m}^n \d\x \d t \le C {\|\grad_m^n \bphi_m^n\|}_{L^\infty(Q_T)};
\ee
\end{enumerate}
Then, up to a subsequence, one has
$$
\iint_{Q_T} \omega_m  \pi_m^n \u_m^n  \pi_m^n \bv_m^n \varphi \d\x\d t \underset{m,n\to \infty}{\longrightarrow}
\iint_{Q_T} \omega u v \varphi \d\x \d t, \qquad \forall \varphi \in \Dd(Q_T).
$$
\end{prop}

Combining Proposition~\ref{thm:compcomp-disc} with Proposition~\ref{prop:Young}
we get the following result, which is the main result of our paper.
\begin{thm}\label{thm:main}
Keeping the assumptions of Proposition~\ref{thm:compcomp-disc}, we additionally suppose that Assumption~$({\bf A}_\x2)$
holds, and that there exists a maximal monotone graph $\beta$ such that
$\bv_m^n \in \beta(\u_m^n)$ for all $m, n \ge 1.$
Then the weak limits $u,v$ of $\left(\pi_m^n \u_m^n\right)_{m,n}$ and  $\left(\pi_m^n \bv_m^n\right)_{m,n}$
satisfy $u \in \beta(v)$ for a.e. $(x,t) \in Q_T$. Moreover, up to an unlabeled subsequence,
\begin{itemize}
   \item[(i)] If $\beta$ is single-valued, then, $\pi_m^n\bv_m^n\to v$ a.e. in $Q_T$ as $m,n\to \infty$.
   \item[(ii)] If $\beta^{-1}$ is single-valued, then $\pi_m^n \u_m^n\to u$ a.e. in $Q_T$ as $m,n\to \infty$.
 \end{itemize}
\end{thm}
With Propositions~\ref{thm:compcomp-disc} and~\ref{prop:Young} at hand,
the proof of Theorem~\ref{thm:main} is straightforward.
Indeed, Assumption~$({\bf A}_\x2)$ ensures that
$$
\bv_m^n \in \beta(\u_m^n)\quad  \text{(i.e. $v_{m,i}^{n,k} \in \beta(u_{m,i}^{n,k})$)}
\quad \implies  \quad
\pi_m^n \bv_m^n \in \beta(\pi_m^n \u_m^n) \text{ a.e. in } Q_T,
$$
so that one can directly use Proposition~\ref{prop:Young}.
The remaining of this section will be devoted to the proof of Proposition~\ref{thm:compcomp-disc}

\subsubsection{Reduction of the problem to a semidiscrete situation}

In order to prove Proposition~\ref{thm:compcomp-disc}, our strategy consists in reducing
Estimate~\eqref{eq:estim-hdmn} into a semi-discrete estimate that will be easier to handle.
This is the purpose of Lemma~\ref{lem:multi-to-semi} stated and proved below.

\begin{lem}\label{lem:multi-to-semi}
Let $\h \delta_m^n$ be a time-differentiation operator as introduced in~\S\ref{sssec:multi-step}.
Assume that (${\bf A}_\x3$) and (${\bf A}_t$) are fulfilled, and that ~\eqref{eq:estim-hdmn} holds for all
$\bphi_m^n \in \R^{m\times(n+1)}$, then
\be\label{eq:semi}
\iint_{Q_T} \omega_m \pi_m^n \u_m^n \p_t \varphi \,\d\x \d t \le C \|\grad \varphi\|_{L^\infty(Q_T)}, \qquad
\forall \varphi \in \Dd(Q_T).
\ee
\end{lem}
\begin{Proof}
Let $\bphi_m^n \in \R^{m\times(n+1)}$ be arbitrary, then define $\h\bphi_m^n = {(\h\A_n^{-1})}^T \bphi_m^n$, so that,
for all $k\in\{1,\dots, n\}$, one has
\begin{multline}\label{eq:estim-grad}
{\| \grad_m \h\bphi_m^{n,k} \|}_{L^\infty(\O)} =
\left\| \sum_{\ell = 1}^k \left(\A_n^{-1}\right)_{\ell,k} \grad_m \bphi_m^{n,\ell} \right\|_{L^\infty(\O)}\\
\le \sum_{\ell = 1}^k \left|\left(\A_n^{-1}\right)_{\ell,k}\right| {\| \grad_m \bphi_m^{n,\ell}\|}_{L^\infty(\O)}
\le {\|\A_n^{-1} \|}_{1} {\| \grad_m^n \bphi_m^n \|}_{L^\infty(Q_T)}.
\end{multline}
The link~\eqref{eq:hdmn} between $\h\delta_m^n$ and $\delta_m^n$ provides that
\begin{multline*}
\iint_{Q_T} \omega_m \h\delta_m^n \u_m^n \pi_m^n \h \bphi_{m}^n \d\x \d t =
\iint_{Q_T} \omega_m \pi_m^n \left( \T_n^{-1} \h\A_n \T_n \M_n \u_m^n\right) \pi_m^n \h\bphi_{m}^n \d\x \d t \\
 = \sum_{k=1}^n \dt_{n,k} \int_\O \omega_m \sum_{\ell = 1}^n \frac{{(\A_n)}_{k,\ell}\dt_{n,\ell}}{\dt_{n,k}}
 \pi_m\left( \frac{\u_m^{n,\ell} - \u_m^{n,\ell-1}}{\dt_{n,\ell}} \right)
 \pi_m \h\bphi_m^{n,k} \d\x\d t \\
 =   \sum_{\ell = 1}^n \dt_{n,\ell} \int_\O \omega_m  \pi_m\left( \frac{\u_m^{n,\ell} - \u_m^{n,\ell-1}}{\dt_{n,\ell}} \right)
 \sum_{k=1}^n {(\A_n^T)}_{\ell,k} \pi_m \h\bphi_m^{n,k} \d\x \d t.
\end{multline*}
Since $\bphi_m^n = \h\A_n^T  \h \bphi_m^n$, one obtains that
$$
\iint_{Q_T} \omega_m \h\delta_m^n \u_m^n \pi_m^n \h\bphi_{m}^n \d\x \d t
= \iint_{Q_T} \omega_m \delta_m^n \u_m^n \pi_m^n \bphi_m^n \d\x\d t,
$$
which ensures together with~\eqref{eq:estim-hdmn}, \eqref{eq:estim-grad} and (${\bf A}_t$) that
\be\label{eq:estim-dmn}
\iint_{Q_T} \omega_m \delta_m^n \u_m^n  \pi_m^n \bphi_{m}^n \d\x \d t \le C {\|\grad_m^n \bphi_m^n\|}_{L^\infty(Q_T)}.
\ee

Let $\varphi \in \Dd(Q_T)$ be arbitrary, and set $\bphi_m^n = - \bP_m^n \varphi$ so that,
thanks to Assumption~(${\bf A}_\x3$), one has
$$
\|\grad_m^n \bphi_m^n\|_{L^\infty(Q_T)} \le C \|\grad \varphi\|_{L^\infty(Q_T)}, \qquad \forall m,n\ge 1.
$$
Thus it follows from Lemma~\ref{lem:constability} and~\eqref{eq:estim-dmn} that
$$
\iint_{Q_T} \omega_m \pi_m^n \u_m^n \p_t \varphi \d\x \d t
= \iint_{Q_T} \omega_m \delta_m^n \u_m^n \pi_m^n \bphi_m^n \d\x \d t \le C\|\grad \varphi\|_{L^\infty(Q_T)},
$$
concluding the proof of Lemma~\ref{lem:multi-to-semi}.
\end{Proof}

\subsubsection{Proof of Proposition~\ref{thm:compcomp-disc}}

The proof mimics the one of~\cite[Lemma 3.1]{Moussa14}.
Let $\varphi \in C^\infty_c(Q_T)$,
we denote $\a = {\rm dist}({\rm supp}\, \varphi ; \p \O)$. Let $\rho \in C^\infty_c(\R^d)$ be such that
$\rho(-\x) = \rho(\x) \ge 0,$ for all $\x \in \R^d$, such that  ${\rm supp}\, \rho \subset B_d(0,1)$ and such that $\int_{\R^d} \rho(\x)\d\x = 1.$
For $\ell \in \N$, $\ell > 1/\a$ and $\x \in \R^d$, one defines $\rho_\ell(\x) = \ell^{d} \rho(\ell \x)$, so that ${\rm supp} \rho_\ell \subset B_d(0,1/\ell)$,
$\int_{\R^d} \rho_\ell(\x) = 1$ and $\varphi * \rho_\ell \in C^\infty_c(Q_T)$, where $*$ is the usual convolution w.r.t. the space variable $\x$.

One splits
\begin{multline}\label{eq:R1234}
\iint_{Q_T} \left( \omega uv - \omega_m \pi_m^n \u_m^n \pi_m^n  \bv_m^n\right) \varphi \d\x\d t  \\
= R_1(\ell) + R_2(\ell,m,n) + R_3(\ell,m,n) + R_4(\ell,m,n),
\end{multline}
where
\begin{align*}
R_1(\ell) = &  \iint_{Q_T}  [u \omega v - u ((\omega v)*\rho_\ell)]   \varphi  \d\x\d t, \\
R_2(\ell,m,n) = &  \iint_{Q_T} [ u ((\omega v)*\rho_\ell) - \pi_m^n \u_m^n ( \omega_m \pi_m^n \bv_m^n * \rho_\ell) ]
\varphi \d\x\d t,\\
R_3(\ell,m,n) = &  \iint_{Q_T} [ \pi_m^n \u_m^n (\omega_m \pi_m^n \bv_m^n * \rho_\ell) -
			(\omega_m \pi_m^n \u_m^n \pi_m^n \bv_m^n) * \rho_\ell  ] \varphi \d\x\d t, \\
R_4(\ell,m,n) = &  \iint_{Q_T} [ (\omega_m \pi_m^n \u_m^n \pi_m^n \bv_m^n) * \rho_\ell
			- (\omega_m \pi_m^n \u_m^n \pi_m^n \bv_m^n)  ] \varphi \d\x\d t.
\end{align*}
Clearly, $(\omega v)*\rho_\ell$ tends weakly to $\omega v$ in $L^{q'}((0,T);L^{p'}(\O))$
as $\ell$ tends to $+\infty$, leading to
\be\label{eq:R1}
\lim_{\ell \to \infty} \left| R_1(\ell) \right| = 0.
\ee
Since $\rho_\ell$ is an even function, using the H\"older inequality and the bounds assumed in the statement of the proposition we find
\begin{align}
\left| R_4(\ell,m,n) \right| =& \left| \iint_{Q_T}  (\omega_m \pi_m^n \u_m^n \pi_m^n \bv_m^n)
(  \varphi* \rho_\ell - \varphi) \d\x\d t\right| \nn \\
\le &C \| \varphi * \rho_\ell - \varphi \|_{L^\infty(Q_T)} \le \frac{C}\ell   \| \grad \varphi \|_{L^\infty(Q_T)}.\label{eq:R4}
\end{align}
In particular, $R_4(\ell, m,n)$ tends to $0$ uniformly with respect to $m$ and $n$ as $\ell$ tends towards $+\infty$.
\medskip

Concerning the term $R_3(\ell,m,n)$, one has for almost all $(\x,t) \in Q_T$ :
\begin{multline*}
S_{m,n,\ell}(\x,t) := \pi_m^n \u_m^n(\x,t) [(\omega_m \pi_m^n \bv_m^n)*\rho_\ell](\x,t) -
				[(\omega_m \pi_m^n \u_m^n \pi_m^n \bv_m^n)*\rho_\ell](\x,t) \\
 \le \int_{B_d(0,1/\ell)} \bigg[ \Pi_m^n u_m^n(\x,t) - \Pi_m^n u_m^n(\x-\y, t) \bigg]\omega_m(\x-\y)\Pi_m^n v_m^n(\x-\y,t) \rho_\ell(\y) \d\y.
\end{multline*}
We deduce from Fubini's theorem that
\begin{multline*}
|R_3(\ell,m,n)| \le \| \varphi \|_{L^\infty(Q_T)} \left\| S_{m,n,\ell} \right\|_{L^1(Q_T)}  \\
\le \;  \ov \omega \| \varphi \|_{L^\infty(Q_T)}  \left\| \pi_m^n \bv_m^n \right\|_{L^{q'}(L^{p'})}
 \int_{B_d(0,\ell)} \| \pi_m^n \u_m^n - \pi_m^n \u_m^n(\cdot - \y,\cdot) \|_{L^q(L^p)}  \rho_\ell(\y) \d\y.
\end{multline*}
It results from Assumption~\eqref{eq:translates_x} and from the assumption  $\| \u_m \|_{p,m,q,n} \le C$ that
$$
 \int_{B_d(0,\ell)} \| \pi_m^n \u_m^n - \pi_m^n \u_m^n(\cdot - \y,\cdot) \|_{L^q(L^p)}  \rho_\ell(\y) \d\y
 \underset{\ell \to +\infty}{\longrightarrow} 0
$$
uniformly w.r.t $m$ and $n$, so that
\be\label{eq:R3}
R_3(\ell,m,n) \underset{\ell \to +\infty}{\longrightarrow} 0\quad  \text{ uniformly w.r.t $m$ and $n$}.
\ee

Let us now focus on controlling the term $R_2(\ell,m,n)$. Fix $\ell \ge 1$, and define the family
$
\left(z_{m,n,\ell}\right)_{m,n} \subset L^\infty(Q_T)
$
by
$$
z_{m,n,\ell} :=(\omega_m \pi_m^n \bv_m^n)\ast \rho_\ell, \qquad \forall m,n\ge 1.
$$
The sequence $(\omega_m \pi_m^n \bv_m^n)_{m,n}$ is uniformly bounded
in $L^1(Q_T)$, then it follows from the regularity of $\rho_\ell$ that there exists $C_\ell>0$ (possibly depending on $\ell$,
but neither on $m$ nor on $n$)
such that
\be\label{eq:BV-space}
\left\|\grad z_{m,n,\ell}\right\|_{L^1(Q_T)}\le C_\ell, \qquad \forall m,n\ge1.
\ee
On the other hand, given $\psi \in C^\infty_c(Q_T)$, we have
\begin{align*}
\langle \p_t z_{m,n,\ell}, \psi \rangle_{\Dd'(Q_T),\Dd(Q_T)}= & -  \iint_{Q_T} z_{m,n,\ell} \p_t \psi \d\x\d t \\
= & -  \iint_{Q_T}\omega_m \pi_m^n \bv_m^n \p_t (\psi \ast \rho_\ell) \d\x \d t.
\end{align*}
Therefore, it follows from Lemma~\ref{lem:multi-to-semi} that
\begin{equation}\label{eq:duality-bound}
 \left| \langle \p_t z_{m,n,\ell}, \psi \rangle_{\Dd'(Q_T),\Dd(Q_T)} \right| \le C \|\grad(\psi\ast\rho_\ell)\|_{\infty}
\le C_\ell \|\psi\|_{C([0,T];L^1(\O))} 
\end{equation}
for some $C_\ell$ possibly depending on $\ell$ but neither on $m$ nor on $n$.
As a consequence, for fixed $\ell$, the sequence
$\left(\p_t z_{m,n,\ell}\right)_{m,n}$ is bounded in the space of finite Radon measures on $Q_T$.
Along with~\eqref{eq:BV-space} this ensures that the family
$\left(z_{m,n,\ell}\right)_{m,n}$ is bounded in the space $\BV(Q_T)$,
thus yielding
\be\label{eq:compact-L1}
\text{ $\left(z_{m,n,\ell} \right)_{m,n\ge 1}$ is relatively compact in $L^1(Q_T)$. }
\ee

Since $z_{m,n,\ell}$ is piecewise constant in time and smooth in space, the map 
$$V_{m,n,\ell}: \begin{cases}
\O \to \R_+\\[5pt] 
\displaystyle \x \mapsto \int_{(0,T)} |\p_t z_{m,n,\ell}(\x,t)|
\end{cases}$$
is continuous on $\O$. Let $\x_{m,n,\ell} \in \O$ be such that 
$$V_{m,n,\ell}(\x_{m,n,\ell}) \ge \frac12 \sup_{\x \in \O} V_{m,n,\ell}.$$
Let $\varphi \in C^\infty_c((0,T))$, and let $k \in \N$ with $k \ge 1/d(\x_{m,n,\ell},\p\O)$, then 
choosing $\psi: (\x,t) \mapsto \varphi(t) \rho_k (\x - \x_{m,n,\ell})$ in~\eqref{eq:duality-bound} and letting $k$ tend to $+\infty$
yields 
$$
V_{m,n,\ell}(\x_{m,n,\ell}) \le C_\ell, \qquad \forall m,n \ge 1, 
$$
whence 
\be\label{eq:BVt-Linf}
\| V_{m,n,\ell} \|_{L^\infty(\O)} \le C_\ell, \qquad \qquad \forall m,n \ge 1, 
\ee
where $C_\ell$ depends on $\ell$ but neither on $m$ nor on $n$.
In addition, the family $\left(\omega_m \pi_m^n \bv_m^n\right)_{m,n}$ being bounded in 
$L^1(Q_T)=L^1\bigl(\O;L^1((0,T))\bigr)$, the definition of $z_{m,n,\ell}$ ensures that 
$$
{\|z_{m,n,\ell}\|}_{C(\O;L^1((0,T)))} \le C_\ell, \qquad \forall m,n \ge 1.
$$
Therefore, for all $\x \in \O$, there exists $t^\star_{\x,m,n,\ell} \in (0,T)$ (possibly depending on 
$\x$, $m$, $n$ and $\ell$) such that 
\be\label{eq:CxL1t}
|z_{m,n,\ell}(\x,t^\star_{\x,m,n,\ell})| \le C_\ell, \qquad \forall m,n \ge1.
\ee
For all $\x \in \O$ and almost all $t \in (0,T)$, one has 
$$
|z_{m,n,\ell}(\x,t)|  \le   |z_{m,n,\ell}(\x,t^\star_{\x,m,n,\ell})| + V_{m,n,\ell}(\x) \le C_\ell
$$
thanks to~\eqref{eq:BVt-Linf} and~\eqref{eq:CxL1t}, whence
\be\label{eq:z-Linf}
\|z_{m,n,\ell}\|_{L^\infty(Q_T)} \le C_\ell, \qquad \forall m,n\ge 1
\ee
for some $C_\ell$ depending neither on $m$ nor on $n$.
Due to~\eqref{eq:z-Linf} and~\eqref{eq:compact-L1}, we deduce that the family
$\left(z_{m,n,\ell} \right)_{m,n\ge 1}$ is relatively compact in $L^r(Q_T)$ for all $r \in [1,+\infty)$.
In addition, (up to an unlabeled subsequence) $z_{m,n,\ell}$ converges weakly in $L^{q'}((0,T);L^{p'}(\O))$, 
as $m,n\to\infty$, to the limit  $(\omega v)\ast\rho_\ell$.
Choosing $r=\max\{p',q'\}$ we see that, up to an unlabeled subsequence,
 $z_{m,n,\ell}$ converges towards $(\omega v)\ast\rho_\ell$
 strongly in $L^{q'}((0,T);L^{p'}(\O))$ as
$m$ and $n$ tend towards $+\infty$.
Therefore, since $\pi_m^n \u_m^n$ converges weakly in $L^q((0,T);L^p(\O))$ towards $u$, we can claim that
\be\label{eq:R2}
\lim_{m,n\to \infty} |R_2(\ell,m,n)| = 0, \qquad \forall \ell \ge 1.
\ee

Let $\eps>0$, then using~\eqref{eq:R1}, \eqref{eq:R3} and \eqref{eq:R4}, there exists $\ell_\eps\ge 1$
such that
$$
| R_1(\ell_\eps) | +| R_3(\ell_\eps,m,n) | + | R_4(\ell_\eps,m,n) | \le \eps, \qquad \forall m,n \ge 1.
$$
It follows from~\eqref{eq:R1234} and \eqref{eq:R2} that
$$
\limsup_{m,n\to \infty} \left| \iint_{Q_T} \left(\omega_m \pi_m^n \u_m^n \pi_m^n  \bv_m^n-
 \omega uv\right) \varphi \d\x\d t \right| \le \eps, \qquad \forall \eps >0,
$$
concluding the proof of Proposition~\ref{thm:compcomp-disc}. \hfill $\square$

\section{Application to Porous Medium equation}
\label{sec:Porous}

The goal of this section is to provide a new convergence result for a two-point flux in space (cf. \cite{EGH00}) and BDF2 in time (see Remark~\ref{rem:BDF2})
approximation of the solution to the porous medium equation as an application of Theorem~\ref{thm:main}.
More precisely, let $\O$ be a polygonal subset of $\R^d$ with outward normal $\n$, let $T$ be a
finite time horizon, let $q > 1$, then given $u_0 \in L^2(\O)$,
we aim to approximate the solution $u$ of
\be\label{eq:porous}
\begin{cases}
\p_t u - \Delta (|u|^{q-1}u) =0 & \text{ in } Q_T = \O \times (0,T), \\
\grad (|u|^{q-1}u ) \cdot \n = 0 & \text{ on } \p \O \times (0,T), \\
u(\cdot, 0) = u_0 & \text{ on } \O.
\end{cases}
\ee
Such an equation has been widely studied in the last decades, see in particular~\cite{Otto01,Vaz07}.
In particular, it is well-known that the problem~\eqref{eq:porous} admits a unique solution $u \in L^\infty((0,T);L^2(\O))$
and such that $u^{(q+1)/2} \in L^2((0,T); H^1(\O))$ (see e.g.~\cite{AL83,Otto96}).
We propose a formally second-order accurate in both time and space Finite Volume scheme, and show the
convergence of the corresponding family approximate solution towards the unique solution $u$ as the discretization
parameter tend to $0$.

\subsection{Discretization of $Q_T$}

We require the spatial mesh to fulfill the so-called \emph{othogonality condition},
sticking to the definition of~\cite[Definition~3.1]{EGH00} for an admissible discretization $(\Tt, \Ee, (\x_K)_{K\in\Tt})$ of $\O$.
More precisely, the domain $\O$ is supposed to be split in a tessellation $\Tt$ of open polygonal convex subsets $\{K\}_{K\in\Tt}$,
such that $\bigcup_{K\in\Tt} \ov K = \ov \O$ and
$K \cap L = \emptyset$ if $K \neq L$, where $(K,L) \in \Tt^2$. Each \emph{control volume} $K \in \Tt$
is endowed of a so-called \emph{center} $\x_K \in K$\footnote{This assumption is made in order to lighten the presentation; it is a classical issue to relax this assumption, requiring only that $\n_{KL}$ point from $K$ to $L$ (this is the case, e.g., under the Delaunay condition for simplicial meshes) and that the mesh size $h$ be defined by $h_\Tt = \max_{K \in \Tt} {\rm diam} (K\cup\{\x_K\})$; we refer, e.g., to \cite{AndrBendahmaneSaad} for details.}, and of interfaces (edges, if $d=2$; faces, if $d=3$) $\s \subset \p K$ contained in hyperplanes of $\R^{d-1}$, the
$(d-1)$-dimensional Lebesgue measure of an interface $\s$ being supposed to be strictly positive.
The intersection of the closure of two elements $K,L \in \Tt$ is either an interface (denoted by $\s_{KL}$), or a manifold of dimension less than $d-1$, or it is empty; we define the set $\Ee$ of the interfaces by $\Ee = \{ \s_{KL} \; | \; K,L \in \Tt\}$.
For all $K \in \Tt$, we denote by $\Nn_K\subset \Tt$ the set of the neighboring cells of $K$, defined by $L \in \Nn_K$ iff $\s_{KL} \in \Ee$.
We assume that for all $K \in \Tt$ and $L \in \Nn_K$, the segment $[x_K,x_L]$ crosses the interface $\s_{KL}$ orthogonally.
We denote by $\n_{KL} = \frac{\x_K - \x_L}{|\x_K - \x_L|}$ the unit normal vector to $\s_{KL}$ outward with respect to $K$ and inward with
respect to $L$.
Note that the above definition of $\Ee$ does not contain the interfaces lying on the boundary $\p\O$ of $\O$. We define by $\Ee_{\rm ext}$ the
set of such boundary interfaces, and $\Ee_{K, \rm ext}$ the subset of $\Ee_{\rm ext}$ made of the interfaces of $\p\O$.
In what follows, we denote by $m_K$ the $d$-dimensional Lebesgue measure of $K\in\Tt$, by $m_{KL}$ the $(d-1)$-dimensional
Lebesgue measure of $\s_{KL} \in \Ee$, and by $m_\s$ the $(d-1)$ dimensional measure of $\s \in \Ee_{\rm ext}$.
We introduce the size $h_\Tt$ and the regularity $\rho_\Tt$
of the mesh $\Tt$ by setting
\be\label{eq:h-rho}
h_\Tt = \max_{K \in \Tt} {\rm diam} (K),
\quad \rho_\Tt = \max_{K\in\Tt} \sum_{L\in\Nn_K} \left( \frac{m_{KL} |\x_K - \x_L|}{m_K} + \frac{{\rm diam}(K)}{|\x_K - \x_L|} \right).
\ee
\medskip
Concerning the time discretization of $(0,T)$, in order to simplify the presentation we restrict our study to the simple case of uniform time discretizations.
Given $n  \in \N^\ast$, we denote $\dt  = T/n$ and $t_k = k \dt$ for $k \in \{0,\dots, n\}$.

\medskip
Given $\u_\Tt = (u_K)_{K\in\Tt} \in \R^{\#\Tt}$, we denote by $\pi_\Tt \u_\Tt$ the piecewise constant function
defined almost everywhere in $\O$ by
\be\label{eq:PiTt}
\pi_\Tt\u_\Tt(\x) = u_K \quad \text{ if } \x \in K.
\ee
For $\u_\Tt^n = (u_K^k)_{K\in\Tt, k \in \{0,\dots, n\}} \in \R^{\#\Tt\times(n+1)}$, and $k \in \{0,\dots, n\},$ we denote by
$$
\u_\Tt^{n,k} = (u_K^k)_{K\in\Tt} \in \R^{\#\Tt}.
$$
We denote by $\pi_\Tt^n \u_\Tt^n$ the
piecewise constant function defined almost everywhere in $Q_T$ by
$$
\pi_\Tt^n \u_\Tt^n(\x,t) = \pi_\Tt  \u_\Tt^{n,k}(\x) = u_K^k \quad \text{ if } \x \in K \text{ and } t \in (t_{k-1},t_k].
$$

\medskip
In order to reconstruct a discrete gradient, we introduce the so-called \emph{diamond cells} $D_{KL}$ for $\s_{KL} \in \Ee$, which
are open subsets of $\O$ defined as the convex hull of $\x_K, \x_L$ and $\s_{KL}$.
Given $\bv_\Tt = (v_K)_{K\in\Tt} \in \R^{\#\Tt}$, we define
\be\label{eq:GTt}
\grad_\Tt \bv_\Tt (\x) =  d \frac{(v_K - v_L) }{|\x_K - \x_L|}\n_{KL}, \qquad \forall \x \in D_{KL}, \; \forall \s_{KL} \in \Ee.
\ee
Using the geometrical identity
\be\label{eq:geo-diamond}
{\rm meas}(D_{KL}) = \int_{D_{KL}}\d\x = \frac{m_{KL} |\x_K - \x_L|}d, \qquad \forall s_{KL}\in\Ee,
\ee
we obtain that
\be\label{eq:norm-grad}
\left\|\grad_\Tt \bv_\Tt\right\|_{L^2(\O)^d}^2 = d \sum_{\s_{KL} \in \Ee} \tau_{KL} (v_K - v_L)^2,
\ee
where $\tau_{KL} = \frac{m_{KL}}{|\x_K-\x_L|}$.
This leads to the following definition of the discrete norm $\|\cdot \|_{\x,\Tt}$: $\forall \bv_\Tt = \left(v_K\right)_{K\in\Tt} \in \R^{\#\Tt}$,
\begin{align*}
\| \bv_\Tt \|_{2,\Tt}^2  =& \| \pi_\Tt \bv_\Tt \|_{L^2(\O)}^2 +  \| \grad_\Tt \bv_\Tt \|_{L^2(\O)^d}^2 \\
=& \sum_{K\in\Tt} \left(v_K\right)^2 m_K + d \sum_{\s_{KL}\in\Ee} \tau_{KL} (v_K - v_L)^2.
\end{align*}
We also define $\grad_\Tt^n: \R^{\#\Tt\times (n+1)} \to L^\infty(Q_T)$ by
$$
\grad_\Tt^n \bv_\Tt^n(\cdot,t) = \grad_\Tt \bv_\Tt^{n,k} \;\; \text{ if } t \in (t_{k-1}, t_k], \quad \forall \bv_\Tt^n  \in \R^{\#\Tt \times (n+1)}.
$$

\subsection{The Finite Volume scheme}
The initial data $u_0$ is discretized into $\u_\Tt^{n,0} = \left(u_K^0\right)_{K\in\Tt} \in \R^{\#\Tt}$, where
\be\label{eq:uK0}
u_K^0 = \frac1{m_K} \int_K u_0(\x) \d\x, \qquad \forall K \in \Tt.
\ee
It follows from Jensen's inequality that
\be\label{eq:u0L2}
\| \pi_\Tt \u_\Tt^0 \|_{L^2(\O)}^2 = \sum_{K\in\Tt} \left(u_K^0\right)^2 m_K \le \|u_0\|_{L^2(\O)}^2.
\ee
In the sequel, we denote
$$\psi:\begin{cases} \R \to \R \\ u \mapsto |u|^{q-1} u.\end{cases}$$
We use the implicit Euler scheme for determining
$\u_\Tt^{n,1} = \left(u_K^1\right)_{K\in\Tt}$, i.e.,
\be\label{eq:u1}
\frac{u_K^1 - u_K^0}\dt m_K + \sum_{L \in \Nn_K} \tau_{KL} \left(\psi(u_K^1) - \psi(u_L^1) \right) = 0, \qquad \forall K \in \Tt.
\ee
As soon as $2 \le k\le n$, we use the so-called BDF2 scheme for determining $\u_\Tt^{n,k} = \left(u_K^k\right)_{K\in\Tt}$,
that is required to fulfill: $\forall K \in \Tt, \; \forall k \in \{2,\dots, n\}$,
\be\label{eq:uk}
\frac{\frac32 u_K^k - 2 u_K^{k-1} + \frac12 u^{k-2}_K}\dt m_K +
	 \sum_{L \in \Nn_K} \tau_{KL} \left(\psi(u_K^k) - \psi(u_L^k) \right) = 0.
\ee

\subsection{Main {\em a priori} estimates and existence of a discrete solution}\label{ssec:main-estimates}

Define the function $\phi:\R \to \R$ by
\be\label{eq:phi-porous}
\phi(u) = \int_0^u \sqrt{\psi'(a)} \d a = \frac{2\sqrt{q}}{q+1}|u|^{\frac{q-1}2} u, \qquad \forall u \in \R,
\ee
then  the Cauchy-Schwarz inequality yields
\be\label{eq:CS-phi}
(a-b)(\psi(a)-\psi(b)) \ge \left(\phi(a)-\phi(b)\right)^2, \qquad \forall (a,b) \in \R^2.
\ee
Therefore, multiplying~\eqref{eq:u1} by $\dt u_K^1$ and summing over $K \in \Tt$, using
$$
(a-b)a \ge \frac{a^2}2 - \frac{b^2}2, \qquad \forall (a,b) \in \R^2,
$$
and the classical summation-by-parts, we find
\be\label{eq:NRJ1}
\frac12\sum_{K\in\Tt} \left(u_K^1\right)^2  m_K + \dt \sum_{\s_{KL}} \tau_{KL} \left(\phi(u_K^1) - \phi(u_L^1)\right)^2 \le \frac12\sum_{K\in\Tt} \left(u_K^0\right)^2m_K.
\ee
In order to obtain an estimate on the following time steps, we use the inequality
$$
\left(\frac32 a - 2b + \frac12c\right)a \ge \frac14\left( a^2 + (2a-b)^2 - b^2 - (2b-c)^2\right), \quad \forall (a,b,c) \in \R^3.
$$
Multiplying~\eqref{eq:uk} by $\dt u_K^k$ and summing over $K\in\Tt$ and $k \in \{2,\dots,\ell\}$
for some $\ell\in \{2,\dots, n\}$, we get
\begin{multline}\label{eq:NRJell}
\frac14\sum_{K\in\Tt} \left(\left(u_K^{\ell}\right)^2 + \left(2 u_K^{\ell} - u_K^{\ell-1}\right)^2\right)  m_K
+ \sum_{k=2}^{\ell} \dt \sum_{\s_{KL}} \tau_{KL} \left(\phi(u_K^k) - \phi(u_L^k)\right)^2 \\
\le \frac14\sum_{K\in\Tt} \left(\left(u_K^{1}\right)^2 + \left(2 u_K^{1} - u_K^{0}\right)^2\right)  m_K.
\end{multline}
Combining \eqref{eq:NRJ1} and~\eqref{eq:NRJell} and using~\eqref{eq:u0L2} we find that for all $\ell \in \{1,\dots,n\} $,
\be\label{eq:NRJell2}
\frac14\sum_{K\in\Tt} \left(u_K^{\ell}\right)^2 m_K + \sum_{k=1}^\ell  \dt \sum_{\s_{KL}} \tau_{KL} \left(\phi(u_K^k) - \phi(u_L^k)\right)^2
\le 2 \|u_0\|_{L^2(\O)}^2,
\ee
leading to the following statement.
\begin{prop}\label{prop:NRJ}
Let $\u_\Tt^n$ be a solution of the scheme~\eqref{eq:u1}--\eqref{eq:uk}, then
there exists $C$ depending only on $u_0$ and $d$ (but not on the discretization) such that
$$
\left\| \pi_\Tt^n \u_\Tt^n\right\|_{L^\infty((0,T);L^2(\O))} + \left\| \grad_\Tt^n \phi(\u_\Tt^n)\right\|_{L^2(Q_T)^d} \le C.
$$
\end{prop}

In order to approximate the solution to the scheme~\eqref{eq:uK0}, \eqref{eq:u1} and~\eqref{eq:uk}, one can use
the iterative algorithm based on monotonicity proposed in~\cite[Remark 4.9]{EGH00}, that converges towards a solution
to the scheme. Therefore, there exists at least one solution to the scheme. The uniqueness of the discrete solution follows
from a classical monotonicity property of two-point flux approximation (see e.g.~\cite{Tipi}), leading to the following statement.
\begin{prop}\label{prop:existence}
Let $\left(\Tt,\Ee,(\x_K)_{K\in\Tt}\right)$ be an admissible discretization of $\O$, then
there exists a unique solution to the scheme~\eqref{eq:uK0}, \eqref{eq:u1} and~\eqref{eq:uk}.
\end{prop}

We need another \emph{a priori} estimate before applying Theorem~\ref{thm:main}. This is the purpose of the following
statement.

\begin{lem}\label{lem:flux-L1}
Let $\u_\Tt^n$ be the unique solution to the scheme~\eqref{eq:uK0}, \eqref{eq:u1} and~\eqref{eq:uk}, then
there exists $C$ depending on $\O$, $T$, $q$, $\rho_\Tt$ and $d$ such that
$$
\left\|\pi_\Tt^n \psi(\u_\Tt^n) \right\|_{L^1(Q_T)} + \left\| \grad_\Tt^n \psi(\u_\Tt^n) \right\|_{L^1(Q_T)^d} \le C.
$$
\end{lem}
\begin{proof}
First of all, by adapting to the discrete framework the technical lemma~\cite[Lemma A.1]{Igb07}, we can claim that there exists
$C$ depending only on $u_0$, $q$, $T$ and $d$ such that
\be\label{eq:Igbida}
\left\| \pi_\Tt^n \phi(\u_\Tt^n) \right\|_{L^2(Q_T)} \le C.
\ee
Indeed, in the continuous case  this property relies on the Poincar\'e-Wirtinger inequality,
for which there exist discrete counterparts~\cite{GG10, ABRB11}.
It follows from the definition~\eqref{eq:phi-porous} of the function $\phi$ that there exists $C$  depending only on $u_0$, $q$, $T$ and $d$
such that
\be\label{eq:uLq+1}
\left\| \pi_\Tt^n \u_\Tt^n \right\|_{L^{q+1}(Q_T)} \le C.
\ee

The definition of the discrete gradient $\grad_\Tt^n$ and the geometrical identity~\eqref{eq:geo-diamond} ensure that
\begin{align}
\label{eq:gradPsi-L1}\left\| \grad_\Tt^n \psi(\u_\Tt^n) \right\|_{L^1(Q_T)}
	= & \sum_{k=1}^n \dt \sum_{\s_{KL} \in \Ee} m_{KL} \left| \psi(u_K^k) - \psi(u_L^k) \right| \\
	= & \sum_{k=1}^n \dt \sum_{\s_{KL} \in \Ee} m_{KL} \eta_{KL}^k \left| \phi(u_K^k) - \phi(u_L^k) \right|, \nonumber
\end{align}
where, for all $k \in \{1,\dots, N\}$ and all $\s_{KL} \in \Ee$, we have set
$$
\eta_{KL}^k = \begin{cases}
\displaystyle \frac{ \psi(u_K^k) - \psi(u_L^k) }{\phi(u_K^k) - \phi(u_L^k)} & \text{ if } u_K^k \neq u_L^k, \\[10pt]
\phi'(u_K^k) = \sqrt{q} |u_K^k|^{\frac{q-1}2} & \text{ if } u_K^k = u_L^k.
\end{cases}
$$
In particular, the mean-value theorem yields: $\forall \s_{KL} \in \Ee$, $\forall k \ge 1$,
$$
0 \le \eta_{KL}^k \le \sqrt{q} \max\left\{|u_K^k|^{\frac{q-1}2}, |u_L^k|^{\frac{q-1}2}\right\}
\le \sqrt{q} \left(|u_K^k|^{\frac{q-1}2}+ |u_L^k|^{\frac{q-1}2}\right).
$$
Using the Cauchy-Schwarz inequality in~\eqref{eq:gradPsi-L1} we get
\begin{multline*}
\left\| \grad_\Tt^n \psi(\u_\Tt^n) \right\|^2_{L^1(Q_T)} \le \| \grad_\Tt^n \phi(\u_\Tt^n) \|_{L^2(Q_T)}
\left(  \sum_{k=1}^n \dt \sum_{\s_{KL} \in \Ee} m_{KL} \left(\eta_{KL}^k\right)^2 |\x_K - \x_L| \right)\\
\le C   \sum_{k=1}^n \dt \sum_{\s_{KL} \in \Ee} m_{KL} \left( |u_K^k|^{q-1}+ |u_L^k|^{q-1}\right)  |\x_K - \x_L|  \\
\le C  \sum_{k=1}^n \dt \sum_{K \in \Tt} |u_K^k|^{q-1} \left(\sum_{L \in \Nn_K} m_{KL} |\x_K - \x_L|\right).
\end{multline*}
Using the regularity $\rho_\Tt$ of the mesh defined by~\eqref{eq:h-rho}, one obtains that there exists $C$ depending
only on $u_0$, $T$, $\O$, $q$, $d$, and $\rho_\Tt$ such that
$$
\left\| \grad_\Tt^n \psi(\u_\Tt^n) \right\|^2_{L^1(Q_T)^d} \le C \left\| \pi_\Tt^n \u_\Tt^n \right\|_{L^{q-1}(Q_T)}^{q-1}.
$$
Using~\eqref{eq:uLq+1}, one obtains that
$$
\left\| \grad_\Tt^n \psi(\u_\Tt^n) \right\|_{L^1(Q_T)^d} \le C.
$$
In order to conclude the proof of Lemma~\ref{lem:flux-L1}, it only remains to use again the discrete counterpart of~\cite[Lemma A.1]{Igb07}
to obtain that $\left\|\pi_\Tt^n \psi(\u_\Tt^n) \right\|_{L^1(Q_T)}\le C$.
\end{proof}

\subsection{Compactness of the solution}
This is the point where the main results of this paper are exploited.
Let $(\Tt_m,\Ee_m,(\x_K)_{K\in\Tt_m})_{m\ge1}$ be a sequence of admissible discretizations of $Q_T$ such that
there exists $\rho^\star>0$ satisfying
\be\label{eq:h-rho2}
\lim_{m\to \infty} h_{\Tt_m} = 0, \qquad \sup_{m\ge1}\, \rho_{\Tt_m} \le \rho^\star.
\ee
For the ease of reading, we denote by $\u_m^n \in \R^{\#\Tt_m \times (n+1)}$ instead of $\u_{\Tt_m}^n$ the discrete solution to the
scheme~\eqref{eq:uK0}, \eqref{eq:u1} and~\eqref{eq:uk} corresponding to the mesh $\Tt_m$ and the time step $\dt_n = T/n$.
Similarly, we replace the notations $\pi_{\Tt_m}$, $\pi_{\Tt_m}^n$, $\grad_{\Tt_m}$, and $\grad_{\Tt_m}^n$ by
$\pi_{m}$, $\pi_{m}^n$, $\grad_{m}$ and $\grad_{m}^n$ respectively.

For all $m\ge 1$, the functions
$$
  \u_{m} \mapsto \left\| \pi_{m}\u_{m} \right\|_{L^2(\O)} \quad \text{and} \quad
 \u_{m} \mapsto \| \u_{m} \|_{2,m}:= \left\| \pi_{m}\u_{m} \right\|_{L^2(\O)} + \left\| \grad_{m}\u_{m} \right\|_{L^2(\O)}
$$
define norms on $\R^{\#\Tt_m}$. Let us check the assumptions of \S\ref{ssec:space-discr} on the space discretization.
It follows from the so-called \emph{space-translate estimate}~\cite[Lemma 3.3]{EGH00} that
$$
\lim_{\boldsymbol \xi \to 0} \sup_{m\ge1} \sup_{\bv_{m} \in \R^{\#\Tt_m}}
\frac{\| \pi_{m} \bv_{m}(\cdot + \boldsymbol \xi) - \pi_{m} \bv_{m}\|_{L^2(\O)}}{\|\bv_{m}\|_{2,m}} = 0,
$$
so that~(${\bf A}_\x$\ref{A:compact}) holds.
Since $\pi_m$ consists in the piecewise constant reconstruction, Assumption~$({\bf A}_\x2)$ also holds.
 Finally, observe that $\u_{m} \mapsto \left\| \pi_{m}\u_{m} \right\|_{L^2(\O)}$ defines a
 euclidian norm on $\R^{\#\Tt_m}$. It is easily checked that the linear operator $\bP_{m}$ defined for $\varphi \in C^\infty_c(\O)$
by $\bP_{m}\varphi = \left(\varphi_K\right)_{K\in\Tt_m} \in \R^{\#\Tt_m}$ with
 $$
\varphi_K = \frac1{m_K} \int_K \varphi(\x) \d\x, \qquad \forall K \in \Tt_m,
 $$
  is suitable in Assumption~(${\bf A}_\x$3). Indeed, the discrete gradient of $\bP_{m} \varphi$ is given by
 $$
 \grad_{m}\bP_{m}\varphi (\x) = d \frac{(\varphi_K - \varphi_L)(\x_K - \x_L)}{|\x_K - \x_L|^2},
 \qquad \forall \x \in D_{KL}, \; \forall \s_{KL} \in \Ee_m.
 $$
 Thanks to the continuity of $\varphi$, there exists $(\widetilde \x_K)_{K\in\Tt_m}$ such that $\widetilde \x_K \in K$ and
 $\varphi_K = \varphi(\widetilde \x_K)$ for all  $K \in \Tt_m.$
 Therefore, it is easy to verify that
 $$
 \left\|  \grad_{m}\bP_{m}\varphi \right\|_{L^\infty(\O)} \le d(1+2\rho^\star) \| \grad \varphi \|_{L^\infty(\O)}, \qquad \forall m \ge 1,
 $$
so that Assumption~$({\bf A}_\x3)$ holds true.

Let $\varphi \in C^\infty_c(\O\times[0,T))$, then multiplying the scheme~\eqref{eq:u1} (resp.~\eqref{eq:uk})
by $\dt \varphi_K^{n,1}$ (resp. $\dt \varphi_K^{n,k}$ for $k \ge 2$) where
 $$
 \varphi_K^{n,k} = \int_K \varphi(\x,t_{k-1})\d\x, \qquad \forall K \in \Tt_m, \; \forall n \in \{1,\dots,n\},
 $$
 and summing over $K \in \Tt_m$ and $k \in \{1,\dots,n\}$ we get
$$
\iint_{Q_T} \h\delta_{m}^n \u_{m}^n\; \pi_{m}^n \bP_{m}^n \varphi\; \d\x\d t
	= \frac1d\iint_{Q_T} \grad_m^n \psi(\u_m^n) \cdot \grad_m^n \bP_m^n \varphi \d\x\d t,
$$
 where
$$
\h\delta_{m}^n (\bv_{m}^n)= \pi_{m}^n \left( \T_n^{-1}\h \A_n \T_n \M_n \bv_{m}^n\right), \qquad \forall \bv_m^n \in \R^{\#\Tt_m \times (n+1)},
$$
the matrices $ \T_n$, $\h \A_n$ and $\M_n$ being defined in~\S\ref{sssec:multi-step}.
 The lower triangular matrix $\A_n \in \Mm_n(\R)$ corresponding to the BDF2 method defined
 in Remark~\ref{rem:BDF2} satisfies
 \be\label{eq:An-1_norm}
{ \| (\A_n)^{-1} \|}_1 = \frac32\left(1-\frac1{3^n}\right) \le \frac32, \qquad \forall n \ge 1,
 \ee
so that the time discretization fulfills Assumption~$({\bf A}_t)$.

Finally, observe that thanks to the \emph{a priori} estimates of \S\ref{ssec:main-estimates} and in particular to Lemma~\ref{lem:flux-L1}, we find that there exists $C$ depending only on $u_0$, $\O$, $T$, $q$, $d$,
and $\rho^\star$ in \eqref{eq:h-rho2} such that
\be\label{eq:estim-dt}
\iint_{Q_T} \h\delta_{m}^n \u_{m}^n\; \pi_{m}^n \bP_{m}^n \varphi\; \d\x\d t  \le C \left\| \grad_m^n \bP_m^n \varphi \right\|_{L^\infty(Q_T)},
\quad \forall \varphi \in C^\infty_c(Q_T).
\ee

 With the above ``weak time derivative estimate''~\eqref{eq:estim-dt} 
  and the ``space derivative estimate'' of Proposition~\ref{prop:NRJ} at hand, we can apply Proposition~\ref{thm:compcomp-disc} and Theorem~\ref{thm:main}. We conclude that
 \be\label{eq:a.e.porous}
 \pi_{m}^n \u_{m}^n \underset{m,n\to\infty}\longrightarrow u \quad \text{ a.e. in } Q_T.
 \ee
 Since $\left(\pi_{m}^n \u_{m}^n\right)_{m,n}$ is uniformly bounded both in $L^\infty((0,T);L^2(\O))$
 and in $L^{q+1}(Q_T)$ it follows from Proposition~\ref{prop:NRJ} and~\eqref{eq:uLq+1} that
 \be\label{eq:equi-porous}
 \text{$\left(\pi_{m}^n \u_{m}^n\right)_{m,n}$ is equi-integrable in $L^r((0,T); L^2(\O))$ for all $r \in [1,\infty)$.}
 \ee
Applying Vitali's convergence theorem we deduce the first claim of the following statement.
 \begin{prop}\label{prop:compact}
 Let $\left(\Tt_m, \Ee_m, (\x_K)_{K\in\Tt_m}\right)_{m\ge1}$ be a sequence of admissible discretizations
 of $\O$ such that~\eqref{eq:h-rho2} holds. Let $\left(\u_m^n\right)_{m,n}$ be the corresponding sequence
 of discrete solutions to the scheme~\eqref{eq:uK0}, \eqref{eq:u1}, and~\eqref{eq:uk}, then,
 up to an unlabeled subsequence, there exists $u \in L^\infty((0,T);L^2(\O))$  such that
$$
 \pi_{m}^n \u_{m}^n \underset{m,n\to\infty}\longrightarrow u \quad \text{ strongly in } L^r((0,T);L^2(\O))
 \text{ for all } r\in [1,\infty).
 $$
Moreover, $\phi(u)$ belongs to $L^2((0,T);H^1(\Omega))$ and
  $$
 \grad_m^n \phi(\u_m^n) \underset{m,n\to\infty}\longrightarrow \grad \phi(u) \quad \text{ weakly in } L^2(Q_T)^d. $$
 \end{prop}
 \noindent
 The last claim of Proposition~\ref{prop:compact} is classical. Indeed, it follows from 
  the first claim of the proposition that $\pi_m^n \phi(\u_m^n)\equiv \phi(\pi_m^n \u_m^n)$ converge to $\phi(u)$; in addition, due to the uniform bound of Lemma~\ref{lem:flux-L1}, the discrete gradients $\grad_m^n \phi(\u_m^n)$ converge weakly (up to a unlabelled subsequence) in $L^1(Q_T)$ to some limit $\boldsymbol{w}$. Then the weak limit $\boldsymbol{w}$  of $\Bigl(\grad_m^n \phi(\u_m^n)\Bigr)_{m,n\geq 1}$ can be identified with $\grad\phi(u)$, by passing to the limit in the duality identities that express the action of $\grad \phi(u)$ and $\grad_m^n \phi(\u_m^n)$ on test functions  (see~\cite{EGH00,EGH10} for details of this identification argument).
  Note that one can also prove that
 $$
 \| \grad_m^n \phi(\u_m^n) \|_{L^2(Q_T)^d} \underset{m,n\to\infty}\longrightarrow \sqrt{d} \| \grad \phi(u) \|_{L^2(Q_T)^d},
 $$
prohibiting the strong convergence of $\grad_m^n \phi(\u_m^n)$ towards $ \grad \phi(u)$ if $d \ge 2$.

\subsection{Identification of the limit}

The last step for proving the convergence of the scheme~\eqref{eq:uK0}, \eqref{eq:u1}, and~\eqref{eq:uk}
consists in passing to the limit in the appropriate weak formulation of the scheme, proving that
 the function $u$ exhibited in Proposition~\ref{prop:compact} is the
(unique) weak solution to  problem~\eqref{eq:porous}.

\begin{prop}\label{prop:convergence}
Let $\left(\Tt_m, \Ee_m, (\x_K)_{K\in\Tt_m}\right)_{m\ge1}$ be a sequence of admissible discretizations
 of $\O$ such that~\eqref{eq:h-rho2} holds. Let $\left(\u_m^n\right)_{m,n}$ be the corresponding sequence
 of discrete solutions to the scheme~\eqref{eq:uK0}, \eqref{eq:u1}, and~\eqref{eq:uk}, then
$$
 \pi_{m}^n \u_{m}^n \underset{m,n\to\infty}\longrightarrow u \quad \text{ strongly in } L^r((0,T);L^2(\O))
 \text{ for all } r\in [1,\infty).
 $$
  where $u$ is the unique solution to the problem~\eqref{eq:porous}.
\end{prop}
\begin{proof}
Let $\varphi \in C^\infty_c(\ov \O \times [0,T))$ and let $m,n\ge1$, then define
$\bP_m^n \varphi = \left( \varphi_K^k \right)_{K\in\Tt_m}^{0 \le k \le n}$ by
$$
\varphi_K^k = \frac1{m_K}\int_K \varphi\left(\x, \frac{(k-1)T}n\right)\d\x, \qquad \forall K \in \Tt_m, \; \forall k \in \{1, \dots, n\}.
$$
As in the proof of Lemma~\ref{lem:multi-to-semi}, we introduce the vector
$$\h{\boldsymbol{\varphi}}_m^n = \left( \h \varphi_K^k \right)_{K \in \Tt^m}^{0 \le k \le n} = \left(\h \A_n^{-1}\right)^{T} P_m^n \varphi,$$
the matrix $\h \A_n$ being defined by~\eqref{eq:An} and~\eqref{eq:An-BDF2}.  It is straightforward to verify, thanks to the expression of $\A_n^{-1}$ and to the regularity of $\varphi$, that
\be\label{eq:conv-phi}
\| \grad_m^n \left(\h{\boldsymbol{\varphi}}_m^n - \bP_m^n \varphi\right) \|_{L^\infty(Q_T)^d}  \underset{n\to \infty}\longrightarrow 0, \quad
\text{ uniformly w.r.t. $m$}.
\ee

Multiplying the scheme~\eqref{eq:u1} by $\dt \h\varphi_K^1$ and the scheme~\eqref{eq:uk} by  $\dt \h \varphi_K^k$, then summing
over $K \in \Tt_m$ and $k \in \{1,\dots,n\}$, reorganizing the sums, we find that
\be\label{eq:ABCfinal}
\Aa_m^n(\varphi) + \Bb_m^n(\varphi) + \Cc_m^n(\varphi) = 0,
\ee
where
\begin{align*}
\Aa_m^n(\varphi) = & \iint_{Q_T} \h \delta_{m}^n \u_m^n \pi_m^n \h{\boldsymbol \varphi}_m^n \d\x \d t, \\
\Bb_m^n(\varphi) = &  \sum_{k=1}^n \dt \sum_{K\in \Tt} \psi(\u_K^k) \left(\sum_{L\in\Nn_K} \tau_{KL} (\varphi_K^k - \varphi_L^k)\right), \\
\Cc_m^n (\varphi) = & \frac1d \iint_{Q_T} \grad_m^n \psi(\u^n_m) \cdot \grad_m^n
	\left(\h{\boldsymbol{\varphi}}_m^n - \bP_m^n \varphi\right)\d\x \d t.
\end{align*}
The discrete test-function $\h{\boldsymbol \varphi}_m^n$ has been built in order to ensure that $\mathcal A_m^n(\varphi)$ can be rewritten as
\begin{equation*}
\iint_{Q_T} \delta_{m}^n \u_m^n \pi_m^n \bP_m^n \varphi \d\x \d t = - \iint_{Q_T} \pi_m^n \u_m^n \delta_m^n \bP_m^n \varphi \d\x\d t + \int_\O \pi_m \u_m^{n,0}(\x) \varphi(\x,0) \d\x.
\end{equation*}
Therefore, using classical results (see e.g.~\cite{EGH00}), we can easily check that
\be\label{eq:Afinal}
\Aa_m^n(\varphi) \underset{m,n\to\infty}\longrightarrow - \iint_{Q_T} u \p_t \varphi \d\x \d t - \int_\O u_0 \varphi(\cdot, 0) \d\x, \quad \forall
\varphi \in C_c^\infty(\ov \O \times [0,T)).
\ee
It is now well known (see e.g.~\cite{EGHN98, EGH00} or \cite{AndrBendahmaneSaad}) that
\be\label{eq:Bfinal}
\Bb_m^n(\varphi)
 \underset{m,n\to\infty}\longrightarrow \iint_{Q_T} \grad \psi(u) \cdot \grad \varphi \d\x \d t.
\ee
It follows from Lemma~\ref{lem:flux-L1} and from~\eqref{eq:conv-phi} that
\be\label{eq:Cfinal}
\Cc_m^n(\varphi)
 \underset{m,n\to\infty}\longrightarrow 0.
\ee
Putting~\eqref{eq:Afinal}--\eqref{eq:Cfinal} in~\eqref{eq:ABCfinal}, we find that $u$ is a weak solution of~\eqref{eq:porous}
Finally, as a direct by-product of the uniqueness of the limit value~\cite{Otto96},
one recovers the convergence of the whole sequence  $\pi_m^n \u_m^n$ towards $u$ in $L^2(Q_T)$.
\end{proof}

\subsection*{Acknowledgements}
This work was supported by the French National Research Agency (ANR).
More precisely, Boris Andreianov and Cl\'ement Canc\`es
acknowledge the support of the GEOPOR project (ANR-13-JS01-0007-01),
while Ayman Moussa acknowledges the support of the KIBORD project (ANR-13-BS01-0004).
Boris Andreianov thanks LJLL, Paris 6 and IRMAR, Rennes for the hospitality during the preparation of this paper.

\enlargethispage{\baselineskip}

\end{document}